\documentclass[a4paper,11pt]{amsart}
\usepackage[plainpages=false]{hyperref}
\usepackage{amsfonts,latexsym,rawfonts,amsmath,amssymb,amsthm, mathrsfs, lscape}
\usepackage{verbatim}

\usepackage[all]{xy}
\usepackage{graphicx,psfrag}

\usepackage{array,tabularx}

\usepackage{setspace}

\newtheorem{thm}{Theorem}

\newtheorem{lem}{Lemma}[section]

\newtheorem{prop}{Proposition}[section]
\newtheorem{q}{Problem}[section]

\theoremstyle{remark}
\newtheorem{rmk}{Remark}[section]

\theoremstyle{definition}
\newtheorem{defn}{Definition}[section]                                       

\numberwithin{equation}{section}

\def\p{\partial}
\def\R{\mathbb{R}}
\def\C{\mathbb{C}}

\def\Z{\mathbb{Z}}

\def\l{\lambda}
\def\L{\Lambda}

\def\i{\sqrt{-1}}

\def\cA{{\mathcal A}}

\def\cJ{{\mathcal J}}
\def\cK{{\mathcal K}}

\def\cT{{\mathcal T}}

\def \Ker {\text{Ker}}

\def\la{\langle}
\def\ra{\rangle}
\def\d{d^{*}}

\begin{document}

\title[Evolve nondegenerate two forms]{Evolve nondegenerate two forms}

\author{Weiyong He}

\address{Department of Mathematics, University of Oregon, Eugene, Oregon, 97403}
\email{whe@uoregon.edu}

\maketitle

\section{Introduction}

We propose a method in geometric analysis to study compact manifolds which supports a non degenerate 2-form $\omega$ such that $\omega^n(n!)^{-1}$ defines a volume form on $M$. 
Such a manifold $M$ has to be even dimensional and orientable. When an even dimensional orientable manifold $M$ supports such a two form is well understood in obstruction theory. Indeed a nondegenerate 2-form reduces the structure group of the tangent bundle of $TM$ from $\text{GL}(2n, \R)$ to $\text{Sp}(2n, \R)$, where $\text{Sp}(2n, \R)$ is the group which acts on $\R^{2n}$ isomorphically and preserves the standard symplectic structure on $\R^{2n}$,
\[
dx_1\wedge dy_1+\cdots+dx_n\wedge dy_n.
\]
We identify $\R^{2n}$ with $\C^n$ in the usual way and then $\text{Sp}(2n, \R)$ deformation retracts to its maximal compact subgroup $U(n)\subset \text{GL}(n, \C) (\subset \text{GL}(2n, \R))$. 
Since the latter inclusion is also a homotopy equivalence, this is equivalent to that there is an \emph{almost complex structure} on $M$; in other words, there exists a complex vector bundle structure on the tangent bundle $TM$. By definition, an almost structure $J: TM\rightarrow TM$ is a linear bundle isomorphism covering the identity map of $M$ such that $J^2=-id$.  

Almost complex manifolds contain objects which are central in modern geometry, including symplectic manifolds, complex manifolds and Kahler manifolds for example. A very basic  problem is to ask when an almost complex manifold supports a symplectic structure. The study of symplectic manifolds has witnessed tremendous achievements in last three decades. We refer readers to, for example, \cite{MS} and references therein for more information. 

By definition, a symplectic structure is a smooth manifold with a non degenerate 2-form $\omega$ such that $\omega$ is closed, namely $d\omega=0$.
We should emphasize that there are powerful methods of topological nature (cut-and-paste) in literature to construct symplectic structures. We refer the reader to \cite{Gompf} for example for more details. There are three known sources which obstruct the existence of a symplectic structure. As mentioned above, the manifold has to be almost complex, or equivalently, has a nondegenrate two form. The closeness condition $d\omega=0$ provides another obstruction. Indeed by Stokes' theorem, $[\omega]$ defines a nontrivial de Rham cohomology class in $H^2(M, \R)$ such that $[\omega]^n>0$. The third and last known obstruction is the Seiberg-Witten invariant from gauge theory on 4-manifolds \cite{Taubes}. This last obstruction is rather subtle and we refer the reader to, for example \cite{Gompf2} for  interesting discussions. 

Our motivation is  to propose a  geometric evolution equation, ``canonical" from the point of view of geometric analysis, to study the existence of symplectic forms on a underlying almost complex manifold.  Geometric evolution equations have been studied extensively and  are now a very important subject in geometric analysis with tremendous applications.  In this paper we propose several evolution equations, using the operators $d^*d$ and Hodge Laplacian $\Delta$ which evolve a non degenerate two form $\omega$ (an almost Hermitian structure) to symplectic structures (some canonical models). 
From a geometric analysis (or PDE) point of view, the equation $d\omega=0$ is a first order (linear) system for $\omega$. 
When $M$ is an open manifold (a noncompact manifold without boundary) with a non degenerate two form, Gromov proved that there are always such symplectic two forms by the so-called \emph{h-principle}. When the manifold is compact, the problem is more complicated and there are relatively few tools to handle such a system on compact manifolds. By Lemma \ref{Lchar}, we study the equations for a tamed or compatible pair satisfying $d^*d\omega=0$ or $\Delta \omega=0$. Even though the equation becomes second order nonlinear, we hope that there are more tools from geometric analysis and PDE theory which can help to deal with the problem. 
Indeed, $d^*d\omega=0$ is a degenerate elliptic system and the operator $d^*d$ is essential for our purpose. While Hodge Laplacian $\Delta$ is well studied and it is an elliptic operator when the metric is fixed, but we should mention though $\Delta\omega$ is a rather delicate operator when $\Delta$ is determined by $(\omega, J)$ and $\omega$ and/or $J$ are allowed to vary. 

There are some natural choices  to define an evolution equation to evolve a tamed or compatible pair $(\omega, J)$. The most obvious one might be the \emph{Laplacian flow},
\begin{equation}\label{Lflow}
\frac{\p \omega}{\p t}+\Delta \omega=0.
\end{equation}

The Laplacian flow has already been studied in the context of $G_2$ structure \cite{BX, XY}, where one can already see that even the short time existence is a rather delicate problem, mainly due to the complexity of the linearization of Hodge Laplacian $\Delta$ when viewed as an operator determined by the metric. In our setting similar difficulties arise for \eqref{Lflow} and we shall consider the well-posedness for \eqref{Lflow} elsewhere.  Another  natural choice is 
\begin{equation}\label{Hflow}
\frac{\p \omega}{\p t}+d^*d \omega=0,
\end{equation}
which we shall call a \emph{$d^*d$-flow}, to distinguish with Laplacian operator.   We emphasize that for \eqref{Lflow} and \eqref{Hflow} we fix an almost complex structure $J$ but only require evolving two form $\omega$ is tamed by $J$, while the almost Hermitian metric is determined by $J$ and $J$-invariant part of $\omega$, since neither $\Delta \omega$ nor $d^*d\omega$ is $J$-compatible, even  the initial data $(\omega_0, J_0)$ is assumed to be compatible.
If we want to insist the compatible condition, which may be  more preferable in the point of view of geometric analysis, to evolve an almost Hermitian structure,  we must then allow $(\omega, J)$ to vary simultaneously.  We shall then discuss geometric flows for a compatible pair $(\omega, J)$, with a suitable choice of $K$, 
\begin{equation}\label{Hflow2}
\begin{split}
&\frac{\p \omega}{\p t}+d^*d \omega=0,\\
&\frac{\p J}{\p t}=K.
\end{split}
\end{equation}

Note that the choice of $K$ is not unique in general and this can be a family of flows. 
The most obvious choice is as follows, $K$ can be characterized by
\begin{equation}\label{HK}
\omega(Kx, Jy)=\omega(Jx, Ky)=\frac{1}{2}\left(d^*d\omega(Jx, Jy)-d^*d\omega(x, y)\right).
\end{equation}
We shall still call this flow \emph{$d^*d$-flow} for a compatible pair.  Our first main result is to prove the well-posedness and uniqueness of the $d^*d$-flow for a compatible pair, see Theorem \ref{T-3} and Theorem \ref{T-uniqueness}.

There are other interesting choices for $K$.  Among them we shall discuss $d^*d$-Ricci flow and we believe this flow should also be very interesting. When the initial two form $\omega$ is closed, we show that this condition is preserved, and the flow recovers the anti-complexied Ricci flow studied in \cite{LW}. Another interesting point about $d^*d$-flow and $d^*d$-Ricci flow is that they are closely related to strictly \emph{nearly K\"ahler structures} on dimension six. It seems that these flows can give some approach to evolve almost Hermitian structure (on dimension six) to nearly K\"ahler structures. 

All the considerations can also be discussed for the Laplacian flow. However it seems to be a subtle problem even for the short time existence of Laplacian flow (for either tamed or compatible pair). 
We should mention that geometric flow of a compatible pair $(\omega, J)$ as a flow of almost Hermitian structures has already been studied in literature. For example, 
in  \cite{ST2}, J. Streets and G. Tian proposed a Ricci-flow like system for the compatible pair $(\omega, J)$, called \emph{symplectic curvature flow}, to study the canonical geometric structures on a symplectic manifold.  But our motivation here is completely different. Instead of seeking a canonical geometric structure adapted to a symplectic structure using curvature quantities, we are looking for symplectic structure as a canonical structure over nondegenerate  two forms. We give two proofs of short time existence for \eqref{Hflow2}. One is based on DeTurck's trick and the other is based on Hamilton's theorem. We should emphasize, however, even for short time existence, the system \eqref{Hflow2} seems to be very degenerate since $K$ in \eqref{HK} is completely degenerate in $J$;  at the first glance, $K$ involves no second derivative of $J$.  
And the proof is indeed much more involved.  Nevertheless we will show  that despite the very degeneracy of \eqref{Hflow2}, the degeneracy is \emph{essentially} only caused by the invariance under the action of the diffeomorphism group. We shall consider a more general system for a tamed pair $(\omega, J)$. We can then restore the full parabolicity for such a system (by using DeTurck's trick, in a non-straightforward way). Then we show that the compatibility is preserved by such a system. We show that a solution of this new system gives a solution of \eqref{Hflow2} if $(\omega, J)$ is an almost Hermitian structure. We shall emphasize that the uniqueness does not seem to be a direct extension since the equation on the involved diffeomorphisms is not a parabolic equation (it is a degenerate equation again). Instead we use a theorem proved by Hamilton in his  original proof of short time existence and uniqueness for Ricci flow to give another proof of existence and uniqueness. 

One of our main motivations is to give a precise understanding the existence of symplectic structure among almost complex structures, in particular in dimension four. We believe this is a very important problem and  it is closely related to the geometry and topology of smooth four manifolds. We shall discuss these aspects in a most speculative way, and leave more technical discussions elsewhere. 

We organize the paper as follows. In Section 2 we consider the $d^*d$-flow for a tamed pair $(\omega, J)$ and prove the short time existence and uniqueness. In Section 3 we consider the $d^*d$-flow for an almost Hermitian structure and prove the short time existence and uniqueness. We also prove an extension theorem and consider a simple example with long time existence. In Section 4 we consider the $d^*d$-Ricci flow and its relation to the nearly K\"ahler structures in dimension six. In Section 5 we give some speculative applications of $d^*d$ flow ($d^*d$-Ricci flow) to the geometry of smooth four manifolds.

\section{The $d^*d$-flow for nondegenerate two-form}
In this section we study the the $d^*d$-flow \eqref{Hflow} for nondegenerate two forms. We fix an almost complex structure $J$ on $M$ and consider all nondegenerate two forms which are tamed by $J$. All such forms form a contractible infinite dimensional space (manifold), which we denote by $\cT$. For any initial data $\omega_0\in \cT_J$, we want to study the $d^*d$ flow
\[
\frac{\p \omega}{\p t}+d^*d\omega=0; \;\; \omega(0)=\omega_0. 
\]

We recall some definitions. Let $(M, \omega)$ be a compact manifold with a nondegenerate two form $\omega$. 

\begin{defn}A 2-form $\omega$ is \emph{tamed} by an almost complex structure $J$ if the bilinear form $\omega(\cdot, J\cdot)$ is positive definite on $TM$.  If in addition $\omega(\cdot, J\cdot)$ is symmetric, hence defines an Hermitian structure on $TM$, we say $\omega$ and $J$ are \emph{compatible}; we also denote the almost Hermitian metric by $g(\cdot, \cdot)=\omega(\cdot, J\cdot)$ for a compatible pair $(\omega, J)$. 
\end{defn}

The compatibility condition can also be formulated equivalently  by the condition that $J$ preserves $\omega$ in the sense that $\omega (J\cdot, J\cdot)=\omega(\cdot, \cdot)$.
If $\omega$ and $J$ are compatible, $g$ defines an \emph{almost Hermitian metric} on $M$. We shall note that for a compatible triple $(\omega, J, g)$, any two will determined the other. We shall use $g$ or $(\omega, J)$ to denote the almost Hermitian metric mostly, depending on the structures we want to emphasize. When $\omega$ is only tamed by $J$,  we can also define a metric by \[g(x, y)=\frac{1}{2}\left(\omega(x, Jy)+\omega(y, Jx)\right).\]
Clearly $g$ is compatible with $J$ and hence defines an almost Hermitian metric and its associated non degenerate two form is given by
\[
\tilde \omega (x, y)=g(Jx, y).
\] 
Indeed $\tilde \omega$ is the $J$-invariant part of $\omega$
\[
2\tilde \omega(x, y)=2\omega_{J}=\omega(Jx, Jy)+\omega(x, y).
\]
In general, for a given two form $\beta$, we denote the J-invariant part by
\[
\beta_J(x, y)=\frac{1}{2}\left(\beta(x, y)+\beta(Jx, Jy)\right)
\]
and anti-J invariant part by
\[
\beta_{J_{-}}=\frac{1}{2}\left(\beta(x, y)-\beta(Jx, Jy)\right).
\]
Hence for any tamed pair $(\omega, J)$, we shall also associate the almost Hermitian metric $g$ defined by the compatible pair $(\omega_J, J)$. 
For any given almost Hermitian structure on $M$, 
it induces an inner product structure $\la \cdot, \cdot\ra$ on $\L^{*}M$ and the  Hodge-$*$ operator 
\[
*: \L^*M\rightarrow \L^{*}M
\]
is characterized by, for any $\alpha, \beta\in \L^p(M)$, 
\[
\alpha\wedge *\beta=\la \alpha, \beta\ra  \frac{\omega^n}{n!}. 
\] 
For any $p$-forms $\alpha, \beta\in A^p(M)$, the inner product is defined to be
\[
(\alpha, \beta)=\int_M \alpha\wedge *\beta.
\]
The adjoint operator $d^{*}: A^{p}(M)\rightarrow A^{p-1}(M)$ of exterior differentiation $d$ is then characterized by, 
\[
(\d \alpha, \beta)=(\alpha, d\beta), \forall \alpha\in A^p(M), \beta\in A^{p-1}(M). 
\]
A straightforward computation shows that
\[
\d=-*d*.
\]
The Hodge-Laplacian $\Delta_d$ is defined by
\[
\Delta=\Delta_d=dd^{*}+d^{*}d.\]
The following simple fact gives a characterization of $\omega$ being closed, namely, that $\omega$ is a symplectic form.

\begin{lem}\label{Lchar}For any tamed pair $(\omega, J)$, $\omega$ is closed if and only if $d^*d\omega=0$, where $d^*$ operator is defined by its associate Hermitian metric. For any given Hermitian structure $(\omega, J)$, $\omega$ is symplectic if and only if $\Delta \omega=d^*d\omega=0$. 
\end{lem}

\begin{proof}
Clearly $d^*d\omega=0$ is equivalent to $d\omega=0$. 
And $\Delta \omega=0$ implies that $d\omega=0$ and $d^*\omega=0$. While for the metric determined by a compatible pair $(\omega, J)$, the Hodge star satisfies
\begin{equation}\label{E-dual}
* \omega=\frac{\omega^{n-1}}{(n-1)!}. 
\end{equation}
Since $*$ operator is an algebraic operator, we can work on $T_pM$ for a given point. Then pick up an orthonormal basis $\{e_1, \cdots, e_n, Je_1, \cdots, Je_n\}$ of $T_pM$  with the dual basis $\{e_1^*, \cdots, e_n^*, Je_1^*, \cdots, Je_n^*\}$ of $T_p^*M$. One writes
\[\omega=\sum e_i^* \wedge Je_i^*, g=\sum (e_i^*\otimes e_i^*+Je_i^*\otimes Je_i^*)
\]
The desired identity \eqref{E-dual} is then evident. It follows that $d\omega^{n-1}=0$ if $d\omega=0$; this implies that $d^{*}\omega=-*d*\omega=0$. 
\end{proof}

Fix an almost complex structure $J$. We consider the following energy functional $H_0(\omega)$ and $H_1(\omega)$ for any $\omega$ tamed by $J$,
\begin{equation}
H_0(\omega)=(d\omega, d\omega), H_1(\omega)=(d^*\omega, d^*\omega)
\end{equation} Clearly a symplectic form minimizes $H_0(\omega)$. 
We also consider the harmonic energy 
\[
H(\omega, J)=H_0(\omega)+H_1(\omega)=(d\omega, d\omega)+(\d \omega, \d \omega). 
\]
When $(\omega, J)$ is a compatible pair,  $d\omega=0$ implies $H_0=H_1=0$.
Hence an almost Kahler structure, which by definition is an almost Hermitian structure with $d\omega=0$,  minimizes $H(\omega, J)$.

\subsection{Short time existence} We prove the short time existence of $d^*d$ flow  for any initial data (a tamed pair) in this section. 

\begin{thm}\label{thm-tame}The initial value problem of the $d^*d$ flow $\p_t\omega+d^*d\omega=0$ has a unique smooth solution for a short time  with any initial data $\omega_0\in \cT_J$ at time $t=0$. 
\end{thm}

The principle part (second order) of the linearization of $d^*d\omega$ is given by $d^*d\psi$ (the variation of $\omega$ is given by $\delta\omega=\psi$). We see this directly since for a fixed form $\eta, d^*d \eta$ involves only first order derivative of the metric (hence of $\omega$). This operator is clearly not elliptic and it has an infinite dimensional kernel $\text{Im}d(A^1)\subset A^2$. Hence the $d^*d$ flow is not strictly parabolic. However this equation is not invariant under all diffeomorphisms since we fix the almost complex structure $J$ (the invariant group is a large one, including the diffeomorphisms fixing $J$ but it does not seem to have a good structure). We are not aware that there is a way to restore the full parabolicity  by a gauging fixing trick such as DeTurck's trick for Ricci flow (this causes similar difficulty for Laplacian flow).  R. Hamilton proved a general existence theorem of  a weakly parabolic equation, using Nash-Moser inverse function theorem \cite{Hamilton821} in his seminal paper \cite{Hamilton82}, where the short time existence of Ricci flow flows directly.  Our proof of short time existence of \eqref{Hflow} relies on his result and some basic Hodge theory.  Indeed given the structure of the $d^*d$ flow, the short time existence follows directly from Hamilton's Theorem 5.1, which we record below.

\begin{thm}[Hamilton]\label{thm-h}Let $\p f/\p t =E(f)$ be an evolution equation with integrability condition $L(f)$. Suppose that

(A) $L(f)E(f)=Q(f)$ has at most degree $1$. 

(B) all the eigenvalues of the eigenspaces of $\sigma DE(f)(\xi)$ in $\text{Null}\;\sigma L(f)(\xi)$ have strictly positive real parts.

Then the initial value problem $f=f_0$ at $t=0$ has a unique smooth solution for a short time $0\leq t\leq \epsilon$ where $\epsilon$ may depend on $f_0$. 
\end{thm}

We explain the notations roughly and refer the reader to  \cite{Hamilton82} Section 5 for full details. Here $f$ is a section of a vector bundle $F$ (or belongs to an open set $U$ of $F$) over $M$ and $E: C^\infty(M, U)\rightarrow C^\infty(M, F)$ is a (nonlinear) second order differential operator (viewed as a smooth map in a Frechet space $C^\infty(M, F)$ to itself). $DE(f)\tilde f$ is a linear differential operator in $\tilde f$ of degree 2 and it is the linearization of $E(f)$. We use $\sigma DE(f)$ to denote its symbol (principle part). $L(f) h$ is a differential operator of degree 1 on sections $f\in C^\infty(M, U)$ and $h\in C^\infty(M, F)$ with values in another vector bundle $G$ such that $L(f)E(f)=Q(f)$ only has at most degree 1 in $f$; $L(f)$ is called an \emph{integrability condition} for $E(f)$. By a degree consideration, one sees directly that 
$\sigma L(f)(\xi) \cdot \sigma DE(f) \xi=0$
and in particular, 
\[
\text{Im} \;\sigma DE(f)(\xi)\subset \text{Null}\;\sigma L(f)(\xi).
\]
If $L$ is not trivial then $\sigma DE(f)(\xi)$ must have a null eigenspace; hence $DE(f)$ is not a strictly elliptic operator (or $\p_t-E$ is not a parabolic operator); the best we can hope is that $\sigma DE(f)(\xi)$ is positive when restricted on $\text{Null}\;\sigma L(f)(\xi)$.  Theorem 3 asserts that such a weak parabolic system has a unique smooth short time solution. 

To apply Hamilton's results, we need to find an integrability condition for $d^*d$ flow and verify Condition (B) in the above theorem. Indeed the operator we consider is $E(\omega)=-d^*d\omega$ and the integrability condition is given by $L(\omega)\psi=d^*\psi$, where $d^*$ operator is defined by the metric associated to the pair $(\omega, J)$, as mentioned above. Clearly $L(\omega)E(\omega)=0$. Moreover, the principle part of the linearization of $E(\omega)$ is still $E=-d^*d$, as explained above. By the Hodge decomposition, $A^2=d(A^1)\oplus d^*(A^3)\oplus \text{Ker}(\Delta)$ and hence $\text{Null}(L)=\text{Null}(d^*)=d^*(A^3)\oplus \text{Ker}(\Delta).$ When restricted to $\text{Null}(d^*)$, clearly $E=-d^*d$ is an elliptic operator (it is just the minus Hodge Laplacian $-\Delta$; this negative sign suits exactly for the parabolic equation). Alternatively, one can also verify that  $\sigma DE(\omega)(\xi)$ is positive when restricted on $\text{Null}\;\sigma L(\omega)(\xi)$  as in \cite{Hamilton82} Section 4 by computing the symbol directly. Nevertheless this allows us to use Theorem \ref{thm-h} to conclude Theorem \ref{thm-tame}. 

Even though the $d^*d$ flow makes perfect sense for a tamed pair $(\omega, J)$, it would be more preferable to study a compatible pair $(\omega, J)$ for various reasons. Nevertheless, it suggest $d^*d$ is a degenerate elliptic operator which possesses good properties. This observation gives us the motivation to prove the well-posedness of the $d^*d$ flow for a compatible pair, which we shall consider in the following sections.

\section{The $d^*d$-flow of a compatible pair}\label{HC}

\subsection{The variation of an almost Hermitian structure}
We start with the study of the variation of an almost Hermitian structure of $(\omega, J)$. 
Some properties below are explored in \cite{ST2} for example. Since the consideration is straightforward, we include these discussions here for completeness. 
Suppose the infinitesimal variation of $(\omega, J)$ given by
\[
\delta \omega=\theta, \delta J=K
\]
By the compatibility condition, we have

\begin{prop}The pair of infinitesimal variation  $(\theta, K)$ at $(\omega, J)$ satisfies 
\begin{equation}\label{K-equ}
\begin{split}
JK+KJ&=0\\
\omega(Kx, Jy)+\omega(Jx, Ky)&=\theta(x, y)-\theta(Jx, Jy). 
\end{split}
\end{equation}
\end{prop}

\begin{proof}This is a straightforward computation. The first identity follows from the fact that $J^2=-1$ and the second follows from the fact that $\omega$ is $J$-compatible. 
\end{proof}

The space of elements $(\theta, K)$ satisfying the above can be viewed as the ``tangent space" of $\cA_M$ at $(\omega, J)$ (however we should emphasize that this space is not a linear space). The infinitesimal variations generated by diffeomorphisms will be important for us. 
\begin{prop} For any vector field $X$,  $(L_X\omega, L_XJ)$ satisfies \eqref{K-equ}. In particular, if $(\theta, K)$ is an infinitesimal variation of $(\omega, J)$, then $(\theta+L_X\omega, K+L_XJ)$ is also an infinitesimal variation of $(\omega, J)$. 
\end{prop}

\begin{proof}This is straightforward  and we can understand this geometrically. Let $\phi_t$ be the diffeomorphism generated by $X$ with $\phi_0=id$. Then $(\phi_t^* \omega, \phi_t^*J)$ is an almost Hermitian structure. Now $(L_X\omega, L_XJ)$ is the derivative of $(\phi_t^*\omega, \phi_t^*J)$, hence provides an infinitesimal variation of $(\omega, J)$. In particular, we have
\[
\begin{split}
L_XJ \circ J+J\circ L_XJ=&0\\
\omega((L_XJ)x, Jy)+\omega(Jx, (L_XJ)y)=&L_X\omega(x, y)-L_X\omega(Jx, Jy). 
\end{split}
\]
\end{proof}

We explore some useful properties for such elements $(\theta, K)$. 
Define a two-form $\tilde A$ by
\begin{equation}\label{E-A}
\tilde A(x, y)=\frac{1}{2}\left(\theta(x, y)-\theta(Jx, Jy)\right).
\end{equation}
Clearly $\tilde A$ is anti $J$-invariant, namely $\tilde A(Jx, Jy)+\tilde A(x, y)=0$. 
Let $A\in TM\otimes T^{*}M$ be the unique tensor defined by
\[
g(Ax, y)=\omega (Ax, Jy)=\tilde A(x, y)
\]
Then it  is straightforward to check that $A$ satisfies \eqref{K-equ}; indeed anti $J$-invariance of $\tilde A$ implies $JA+AJ$=0 and
\[\omega(Jx, Ay)=-\omega(Ay, Jx)=-\tilde A(y, x)=\tilde A(x, y).
\]
However $A$ is not the only choice of $K$ satisfying \eqref{K-equ}. Let $\tilde B\in S(T^*M\otimes T^*M)$ be any symmetric two-tensor such that it is anti $J$-invariant.
We also use the notation $\tilde B\in S^{J_{-}}(T^*M\otimes T^*M)$. 
 Let $K$ be the tensor field defined by
\begin{equation}\label{E-pair}
g(Kx, y)=\omega(Kx, Jy)=\tilde A(x, y)+\tilde B(x, y). 
\end{equation}
We have the following characterization of the pair $(\theta, K)$. 
\begin{prop}\label{P-infinitisimal}Any infinitesimal variation $(\theta, K)$ can be characterized by \eqref{E-pair} using $\tilde A$ and $\tilde B$, where $\tilde A$ is the two form defined  in \eqref{E-A} and $\tilde B\in T^*M\otimes T^*M$ is any
symmetric anti $J$-invariant two tensor. \end{prop}

\begin{proof}
First suppose $K$ satisfies \eqref{E-pair}, then $(\theta, K)$ satisfies \eqref{K-equ}. Indeed $JK+KJ=0$ follows from the fact that $\tilde A, \tilde B$ are both anti $J$-invariant. Moreover, 
\[
\omega(Jx, Ky)=-\omega(Ky, Jx)=-\tilde A(y, x)-\tilde B(y, x),
\]
hence we can check that
\[
\omega(Kx, Jy)+\omega(Jx, Ky)=\tilde A(x, y)-\tilde A(y, x)+\tilde B(x, y)-\tilde B(y, x)=2\tilde A(x, y).
\]
This proves that $(\theta, K)$ satisfies \eqref{K-equ}. 

On the other hand, suppose $(\theta, K)$ is an infinitesimal variation of $(\omega, J)$ satisfying \eqref{K-equ}. Write
\[
\omega(Kx, Jy)=\tilde A(x, y)+\tilde B(x, y).
\]
Then we compute
\[
\omega(Jx, Ky)=-\omega(Ky, Jx)=-\tilde A(y, x)-\tilde B(y, x).
\]
Hence by \eqref{K-equ},
\[
\tilde A(x, y)-\tilde A(y, x)+\tilde B(x, y)-\tilde B(y, x)=\theta(x, y)-\theta(Jx, Jy)
\]
It follows that $\tilde B$ is symmetric. Moreover, since $\omega(Kx, Jy)$ and $\tilde A(x, y)$ are both anti $J$-invariant, $\tilde B$ is also anti $J$-invariant. Indeed, $\tilde A$ and $\tilde B$ are the anti-symmetric and symmetric part of $\omega(Kx, Jy)$ respectively, hence are uniquely determined by $(\theta, K)$. 
\end{proof}
 
We write $K=A+B$ with $B$ satisfying
\[
g(Bx, y)=\tilde B(x, y);
\]
as mentioned above, the space of elements $(\theta, \tilde B)$ can be viewed as the tangent space of $\cA_M$ at $(\omega, J)$. In other words, we have naively 
\begin{equation}\label{E-tan}
T_{(\omega, J)}\cA_M\cong \Gamma (\L^2M\oplus S^{J_{-}}(T^*M\otimes T^*M)).
\end{equation}
However we shall not consider the structure of the space of all almost Hermitian structures as an infinite dimensional manifold in a rigorous way and \eqref{E-tan} is only understood in a superficial way, to illustrate Proposition \ref{P-infinitisimal}. \eqref{E-tan} will not be needed for any technical results below.

By the discussion above, there are certain canonical choices of $K$ (or $\tilde B$) to define geometric flow of a compatible pair $(\omega, J)$. We are mainly interested in two cases. One is  the geometric evolution equation for an almost Hermitian structure $(\omega, J)$ with $\tilde B=0$, 
\begin{equation}\label{Hflow3}
\begin{split}
&\frac{\p \omega}{\p t}+d^*d\omega=0\\
&\frac{\p J}{\p t}=K_1(\omega, J),
\end{split}
\end{equation}
where $K_1$ is uniquely determined by
\begin{equation}\label{E3-k1}
g(K_1x, y)=(-d^*d\omega)_{J_{-}}=\frac{1}{2}(d^{*}d\omega(Jx, Jy)-d^{*}d\omega(x, y)). 
\end{equation}
Another interesting  choice of $\tilde B$ is anti-$J$ invariant part of Ricci tensor
\[
\tilde B(x, y)=Ric(Jx, y)+Ric(x, Jy),
\]
which we shall study in the next section.

\subsection{Short time existence}
Now we prove the well-posedness of \eqref{Hflow3} and the short time existence of smooth solution for any initial almost Hermitian structure $(\omega, J)$. Uniqueness will be proved in the next subsection. 

\begin{thm}\label{T-3}For any initial almost Hermitian structure $(\omega_0, J_0)$ at $t=0$, there exists a unique smooth solution of \eqref{Hflow3} with almost Hermitian structure $(\omega(t), J(t))$ for a short time. 
\end{thm}

We use essentially the DeTurck's trick for  the first proof of the existence: modifying by a suitable diffeomorphism we show that \eqref{Hflow3} is equivalent to a strictly parabolic system. However as we mentioned above, this equivalence modulo diffeomorphism is achieved in a very indirect way and it is involved technically. We roughly sketch our strategy. We can consider the modified system, for a vector field $X$,
\begin{equation}\label{Hflow31}
\begin{split}
&\frac{\p \omega}{\p t}+d^*d\omega=L_X\omega\\
&\frac{\p J}{\p t}=K_1(\omega, J)+L_XJ,
\end{split}
\end{equation}
However this system fails to be parabolic in an obvious way. Note that $K_1$ does not involve second derivatives of $J$, while $L_YJ$ (and hence  $K_1+L_YJ$) cannot be an elliptic operator (on $J$) for any choice of  $Y$ (regardless $Y$ involves derivatives of $J$ or not).
Instead we study a system as follows,
\begin{equation}\label{Hflow4-1}
\begin{split}
&\frac{\p \omega}{\p t}=\theta=-d^*d\omega+L_X\omega\\
&\frac{\p J}{\p t}=K=K_2(\omega, J)+L_XJ,
\end{split}
\end{equation}
The system we propose is indeed strictly parabolic (in formal sense) for $(\omega, J)$ by suitable choices of $X$ and $K_2$. 
Here are several key features of $X$ and $K_2$. 
\begin{enumerate}

\item The vector field $X$ involves only first derivatives of $\omega$, and  makes $-d^*d\omega+L_X\omega$ elliptic on $\omega$. But $X$ does not involves first derivatives of $J$, hence $-d^*d\omega+L_XJ$ does not involve second derivatives of $J$.   ( It also explains that we use the operator $-d^*d$ instead of Hodge Laplacian $\Delta$).

\item We have to enlarge our consideration for tame pairs.  This allows us to choose  $K_2$ such that it is an elliptic operator on $J$, hence $K_2+L_XJ$ is also an elliptic operator on $J$ (since $X$ does not involve derivatives of $J$ and $L_XJ$ does not contribute second derivatives of $J$).

\item In general, $K_2\neq K_1$ for a tamed pair $(\omega, J)$. But when $(\omega, J)$ is a compatible pair, we have that $K_2(\omega, J)=K_1(\omega, J)$. 

\end{enumerate}
Then we  apply the standard parabolic theory for the  more general system for a tamed pair $(\omega, J)$ (note that for a compatible pair $\omega$ and $J$, the second derivatives of $\omega$ and $J$ are partially involved) to get a unique smooth short time solution.  
We then show that the compatible condition is preserved if the initial pair is compatible, and prove that $K_2=K_1$ for a compatible pair. Eventually after a gauge transformation induced by $-X$, we show that \eqref{Hflow3} has a short time smooth solution. \\

To begin with, we suppose $\omega$ is a nondegenerate two-form and $J$ is an almost complex structure  such that $(\omega, J)$ is a tamed pair. We fix a background metric $\bar g$ and denote its Levi-Civita connection by $\bar \nabla$. We study the following system for a tamed pair $(\omega, J)$, 
\begin{equation}\label{Hflow4}
\begin{split}
&\frac{\p \omega}{\p t}=-d^*d(\omega_J)+L_X(\omega_J)+g^{pq}\bar \nabla_p\bar \nabla_q (\omega_{J_{-}})\\
&\frac{\p J}{\p t}=K(\omega, J)=K_2(\omega_J, J)+L_XJ,
\end{split}
\end{equation}
where $d^*$ is defined by  $g$, which is the almost Hermitian metric determined by $(\omega_J, J)$. We emphasize that $X$ and $K_2$ are also determined by the almost Hermitian structure  $(\omega_J, J)$.

For simplicity, first we discuss a compatible pair $(\omega, J)$ and determine the corresponding $X$ and $K_2$. We need some preliminary facts, in particular for the operator $d^*d$. Consider an operator $S(X, Y): A^p\rightarrow A^{p-1}$, for any $\eta\in A^p$, 
\[
S(X, Y)\eta=\nabla_X (\iota_Y \eta)-\iota_{(\nabla_X Y)}\eta.
\]
Direct computation shows that $S(X, Y)\eta$ is tensorial in both $X, Y$ (but not in $\eta$).  Indeed we have the following identity, 
\[
d^{*} \eta=-\sum S(e_i, e_i)\eta=-g^{kl}S(X_k, X_l)\eta,
\]
where $\{e_i\}$ is an orthonormal basis and $\{X_k\}$ is a basis such that $g_{kl}=g(X_k, X_l)$.
In particular, for any two-form $\psi=\psi_{ij}dx^i\wedge dx^j$ in a local coordinate ($\psi_{ij}+\psi_{ji}=0$), we have
\begin{equation}\label{E-dstar}
d^{*}\psi=-2g^{kl}\left(\frac{\p \psi_{lj}}{\p x^k}-\psi_{lp}\Gamma^p_{kj}-\psi_{pj}\Gamma^p_{kl}\right) dx^j=-2g^{kl}\psi_{lj, k}dx^j
\end{equation}
Similarly for a three-form $\eta=\eta_{ijk}dx^i\wedge dx^j\wedge dx^k$, we have,
\[
d^*\eta=-3g^{pq}\eta_{pjk, q} dx^j\wedge dx^k,
\]
where the covariant derivative is given by
\[
\eta_{pjk, q}=\frac{\p \eta_{pjk}}{\p x^q}-\eta_{ljk}\Gamma^l_{pq}-\eta_{plk}\Gamma^l_{qj}-\eta_{pjl}\Gamma^l_{qk}
\]
We compute that for a given two form $\psi=\psi_{ij}dx^i\wedge dx^j$, 
\begin{equation}\label{E-harmonic}
\begin{split}
d^*d\psi=&-g^{pq}\left(\left(\nabla_{\p_p} \iota_{\p_q}-\iota_{\nabla_{\p_p}\p_q}\right)\left(\frac{\p \psi_{ij}}{\p x_k}dx^k\wedge dx^i\wedge dx^j\right)\right)\\
=&-g^{pq}\left(\p^2_{p, q}\psi_{ij}+\p^2_{p, i}\psi_{jq}-\p^2_{p, j}\psi_{qi}\right)+O(\p g, \p \psi),
\end{split}
\end{equation}
where $O(\p g, \p \psi)$ denotes the terms of at most degree one and it vanishes in a normal coordinate such that $\p g=0$ at one point. In \eqref{E-harmonic} we drop the factor $dx^i\wedge dx^j$ in the last line. By \eqref{E-harmonic}, we have
\[
-d^*d\omega =g^{pq}\left(\p^2_{p, q}\omega_{ij}+\p^2_{p, i}\omega_{jq}-\p^2_{p, j}\omega_{iq}\right)+O(\p g, \p \omega). 
\]

\begin{prop}\label{P-X}For an almost Hermitian structure $(\omega, J)$, there exists a vector field $X=X(\omega, J, \bar g)$ such that $-d^*d\omega+L_X\omega$ is elliptic on $\omega$ but has no second derivatives of $J$.
\end{prop}
\begin{proof}
We compute
\[
L_X\omega=\left(\frac{\p X^l}{\p x_i}\omega_{lj}-\frac{\p X^l}{\p x_j} \omega_{li}\right)dx_i\wedge dx_j+\iota_X (d\omega).
\]
We assume $X$ does not involve derivatives of $J$. 
Note that $\iota_X (d\omega)$ only involves the terms which has at most first derivatives of $\omega, J$, hence the leading term of $-d^*d\omega+L_X\omega$ reads
\[
\left(\frac{\p X^l}{\p x_i}\omega_{lj}-\frac{\p X^l}{\p x_j} \omega_{li}\right)+g^{pq}\left(\p^2_{p, q}\omega_{ij}+\p^2_{p, i}\omega_{jq}-\p^2_{p, j}\omega_{iq}\right).
\]
There are many choices of $X$ satisfying the properties required. Indeed, let 
\[
A_j=X^l\omega_{lj}-g^{pq}\bar \nabla_p\omega_{qj}.
\]
If we choose  $X$ satisfying  $\p_iA_j-\p_j A_i=O(1)$, where $O(1)$ denotes the terms which involve at most the first derivative of $g, \omega, J$, 
a direct computation gives
\[
-d^*d \omega+L_X\omega= g^{pq}\p^2_{p, q} \omega_{ij}+O(1).
\]
An explicit example is by taking $A_j=0$ and $X$ is determined by 
\begin{equation}\label{E-x}
X^k(\omega, J, \bar g)=g^{pq}\bar \nabla_p (\omega_{qj})\omega^{jk},
\end{equation}
where $\omega^{jk}$ is the inverse of $\omega_{ij}$, namely, it satisfies $\omega_{ij}\omega^{jk}=\delta^k_i$. 
\end{proof}
Next we describe $K_2$.  We define a two form $\psi=\psi_{ij}$ such that
\begin{equation}\label{E-s}
-d^*d\omega=g^{pq}\bar \nabla_p\bar \nabla_q\omega+\psi.
\end{equation}
Then we have
\begin{equation}\label{E-t}
-d^*d\omega+L_X\omega=g^{pq}\bar \nabla_p\bar \nabla_q\omega+\psi+L_X\omega
\end{equation}
Note that $\psi$ involves second derivative of $\omega$ but only first derivative of $J$.  Let $K_2$ be given by
\begin{equation}\label{E-k}
K_2(\omega, J)=\frac{1}{2}g^{pq}\left(\bar \nabla_p\bar \nabla_q J_i^j-\bar \nabla_p\bar \nabla_q (J_a^b) g_{ib}g^{aj}\right)+L^j_i,
\end{equation}
where $L_i^j$ does not involve second derivatives of $J$ and it will be specified as follows.

\begin{prop}\label{P3-k12} For a compatible pair $(\omega, J)$, we can choose $L=L(\omega, J)$  such that \[
K_1(\omega, J)=K_2(\omega, J),
\]
where $K_1$, $K_2$ are given in \eqref{E3-k1} and \eqref{E-k} respectively; in particular $JK_2+K_2J=0$. An important point is that $L$ does not involves second derivatives (or above) of $J$.
\end{prop}

\begin{proof}
For a compatible pair, we have
\[
\omega_{ij}=\omega_{kl}J^k_iJ^l_j.
\]
Hence by taking derivatives, we have
\[
\bar\nabla_p\bar\nabla_q(\omega_{ij})=\bar\nabla_p\bar\nabla_q\left(\omega_{kl}J^k_iJ^l_j\right).
\]
It follows that
\[
\bar\nabla_p\bar\nabla_q(\omega_{ij})-\bar\nabla_p\bar\nabla_q(\omega_{kl})J^k_iJ^l_j=\bar\nabla_p\bar\nabla_q(J^k_i) g_{kj}-\bar\nabla_p\bar\nabla_q(J^l_j) g_{li}+O(1).
\]
We define a two form $\phi$ of type $O(1)$ by
\begin{equation}\label{E-p}
\phi_{ij}=g^{pq}(\bar\nabla_p\bar\nabla_q\omega_{ij}-J^k_iJ^l_j\bar\nabla_p\bar\nabla_q\omega_{kl}-g_{kj}\bar\nabla_p\bar\nabla_qJ^k_i +g_{li}\bar\nabla_p\bar\nabla_qJ^l_j)
\end{equation}
We can then define $L$ by
\begin{equation}\label{E-b}
g(Lx, y)=\psi_{J_{-}}+\frac{1}{2}\phi(x, y).
\end{equation}
Clearly $L^j_i$ does not involve second derivatives of $J$ since $\phi$ is of type $O(1)$ and $\psi$ does not involve second derivative of $J$. By \eqref{E-k} and \eqref{E-p} we get, 
\begin{equation*}\begin{split}
g(K_2\p_i, \p_j)=&\frac{1}{2}\left(g_{jk}\bar\nabla_p\bar\nabla_qJ^k_i-g_{li}\bar\nabla_p\bar\nabla_qJ^l_j\right)+g(L\p_i, \p_j)\\
&=\frac{1}{2}g^{pq}\left(\bar\nabla_p\bar\nabla_q\omega_{ij}-J^k_iJ^l_j\bar\nabla_p\bar\nabla_q\omega_{kl}\right)+\psi_{J_{-}}(\p_i, \p_j)\\
&=(g^{pq}\bar\nabla_p\bar\nabla_q\omega+\psi)_{J_{-}}(\p_i, \p_j)
\end{split}
\end{equation*}
In other words, $K_2$ can be characterized by
\begin{equation}\label{E3-k1-k2}
g(K_2x, y)=\left(g^{pq}\bar\nabla_p\bar\nabla_q\omega+\psi\right)_{J_{-}}=-g(x, K_2 y).
\end{equation}
We can also write the above as
\begin{equation}\label{E3-k22}
\omega(K_2x, Jy)=\omega(Jx, K_2 y)=\left(g^{pq}\bar\nabla_p\bar\nabla_q\omega+\psi\right)_{J_{-}}(x, y)
\end{equation}
Since $-d^*d\omega=g^{pq}\bar\nabla_p\bar\nabla_q\omega+\psi$ by definition of $\psi$, we see that $K_1=K_2$.
\end{proof}

\begin{prop} $K_2$ is an elliptic operator on $J$. Hence the system \eqref{Hflow4-1} is parabolic for $(\omega, J)$ (in formal sense).
\end{prop}
\begin{proof}
We compute the symbol of $K_2$ (on $J$). By taking $g_{ij}=\delta_{ij}$ at one point, we have
\begin{equation}\label{E3-k2}
\sigma K_2(\xi) J^j_i=\frac{1}{2}|\xi|^2 (J_i^j-J^i_j)=|\xi|^2 J_i^j.
\end{equation}
The last equality follows since $J$ is compatible with $g$ and hence $JJ^T=id$ ($g=id$ at the point); this implies that $J+J^T=0$. 
By Proposition \ref{P-X},  the linearized operator of $\theta=-d^*d\omega+L_X\omega$ (on $\omega$) is elliptic, indeed for $\xi=(\xi_1, \cdots, \xi_{2n})$,
\[
\sigma L(\theta, \xi) \beta_{ij}=|\xi|^2 \beta_{ij}.
\]
It is also important to notice that $\theta$ does not involve second derivative of $J$ at all. By \eqref{E-k} and \eqref{E3-k2}, the linearized operator of $K$ (on $J$) is also elliptic (since $X$ does not involve any derivative of $J$, while $B$ involves only first derivative), and we have
\[
\sigma L(K, \xi) P^j_i=|\xi|^2P^j_i.
\]
Note that since $K$ contains second derivative of $\omega$, hence the operator $\sigma L(K) \beta_{ij}$ is not zero. Nevertheless, the whole symbol of the system $(\theta, K)$ on $(\beta, P)$ is given by (take $g_{ij}=\delta_{ij}$ at a point)
\[
\begin{pmatrix} I& 0\\
*& I
\end{pmatrix}
\]
Hence this verifies that \eqref{Hflow4-1} is strictly parabolic in formal sense.  
\end{proof}

Now we are ready to prove that \eqref{Hflow4} has a smooth short time solution for a tamed pair $(\omega, J)$. First note that $(\omega_J, J)$ determines an almost Hermitian structure,  we can then define $X, K_2$ (and $L$) as in \eqref{E-x}, \eqref{E-k} and \eqref{E-b} respectively, using $(\omega_J, J)$  (with $\omega$ replaced by $\omega_{J}$ correspondingly).
With this understanding of $X, K_2$, then we have, 

\begin{prop}There exists a unique smooth short time solution of \eqref{Hflow4}.
\end{prop}
\begin{proof}
We only need to show the system \eqref{Hflow4} is parabolic for a tamed pair $(\omega, J)$.  
By \eqref{E-k} and \eqref{E3-k2}, 
the linearized operator of $K$ (on $J$) is elliptic, and we have
\[
\sigma L(K, \xi) P^j_i= |\xi|^2P^j_i.
\]

By \eqref{Hflow4} and the definition of $X$, the operator $\theta(\omega_J, J)=-d^*d\omega_J+L_X\omega_J$ contributes (to the second derivatives of $\omega$ and $J$)
\[
g^{pq}\bar\nabla_p\bar \nabla_q(\omega_{J}).
\]
Hence for $\theta(\omega_J, J)+g^{pq}\bar\nabla_p\bar \nabla_q (\omega_{J-})$,  the second derivatives of $\omega$ and $J$ involved are of the form, 
\[
g^{pq}\bar\nabla_p\bar \nabla_q(\omega_{J})+g^{pq}\bar\nabla_p\bar \nabla_q(\omega_{J_{-}})=g^{pq}\bar\nabla_p\bar\nabla_q(\omega).
\]
It follows that the linearized operator of $\theta(\omega_{J}, J)+g^{pq}\bar\nabla_p\bar \nabla_q (\omega_{J_{-}})$ (on $\omega$) is elliptic.  It is also important to note that $\theta(\omega_{J}, J)+g^{pq}\bar\nabla_p\bar \nabla_q (\omega_{J-})$ does not involve second derivative of $J$. Now since $K$ is indeed involved with second derivative of $\omega$, hence the operator $\sigma L(K) \beta_{ij}$ is not zero. Nevertheless, the whole symbol of the system $(\theta, K)$ on $(\beta, P)$ is still given by (take $g_{ij}=\delta_{ij}$ at a point)
\[
\begin{pmatrix} I& 0\\
*& I
\end{pmatrix}
\]
Hence this equation is strictly parabolic on a tamed pair $(\omega, J)$. To finish the proof we might quote the standard parabolic theory. 

However there is one more technical point that the space of almost complex structures is not linear due to the fact that $J^2=-id$. This kind of phenomenon can be handled in a rather standard way; for example, a similar situation was considered in \cite{ST2}. The idea is that  we can localize the problem by considering almost complex structures near a fixed complex structure $J_0$. For an almost complex structure $J_0$, the space of almost complex structures (denoted by $\cJ$) near $J_0$ forms a Banach manifold (this can be viewed as a submanifold of the space of all nonsingular endomorphisms of $TM$), and the tangent space is modeled on anti-$J_0$ invariant endomorphisms (which covers the identity of $M$), namely $E: TM\rightarrow TM$ such that $J_0E+EJ_0=0$.  Hence, there exists a diffeomorphism (from a neighborhood of $0$ in $T_{J_0}\cJ$ to a neighborhood $U_{J_0}$ of $J_0$)
\[
\pi: T_{J_0}\cJ\rightarrow U_{J_0}\subset \cJ.
\]
We can then use $\pi$ to pull back almost complex structures in $U_{J_0}$ to the linear space $T_{J_0}\cJ$. Hence as in \eqref{Hflow4}, we can define a flow for a tamed pair $(\omega, E)$ by identifying $J=\pi E$, for $E$ is in a neighborhood of $T_{J_0}\cJ$. Indeed the equation reads, with the initial condition of a tamed pair $(\omega_0, J_0)$ ($E_0=0$),
\begin{equation}\label{E3-l}
\begin{split}
&\frac{\p \omega}{\p t }=-d^*d (\omega_{\pi E})+L_X (\omega_{\pi E})+g^{pq}\bar \nabla_p\bar \nabla_q (\omega_{\pi E_{-}}):=\theta(\omega, E)\\
&\frac{\p E}{\p t}=d\pi^{-1}_{\pi E}\left[K_2(\omega_{\pi E}, \pi E)+L_X (\pi E)\right]:=K(\omega, E),
\end{split}
\end{equation}
where $d\pi^{-1}$ is the tangent map of $\pi^{-1}$ at $\pi E$.  Since $K_2(\omega_{\pi E}, \pi E)+L_X (\pi E)$ is in the tangent space $T_{\pi E}\cJ$, hence $K(\omega, E)$ is well-defined and it is in the image of  $T_E (T_{J_0}\cJ)\cong T_{J_0} \cJ$. 
Hence this defines a flow in the space $\L^2\oplus T_{J_0} \cJ$. 
Note that at time $t=0$, 
the linearized operator of the system \eqref{E3-l} reads exactly the same as in \eqref{Hflow4}. We can then quote the standard parabolic theory to conclude that \eqref{E3-l} has a smooth short time solution. Now denote $J=\pi E$, it is then straightforward to verify that $(\omega, J)$ is a smooth short time solution of \eqref{Hflow4} with the initial tamed pair $(\omega_0, J_0)$.

\end{proof}

Now suppose we have a (short time) unique smooth solution $(\omega(t), J(t))$ of \eqref{Hflow4} with the initial data by an almost Hermitian structure $(\omega_0, J_0)$. We want to show that $(\omega(t), J(t))$ remains to be an almost Hermitian structure. 
\begin{prop}\label{P-c}The compatibility condition is preserved along \eqref{Hflow4}.
\end{prop}

\begin{proof}We need to show that $\omega_{J_-}=0$ along the flow. 
This is clearly a local property on time. We suppose that the smooth solution exists for time $[0, T]$ and we assume that all geometric quantities are uniformly bounded in $[0, T]$, depending only on initial condition and $T$. We claim, \begin{equation}\label{E-c}
\frac{\p }{\p t} \omega_{J_{-}}=\left(g^{pq}\bar\nabla_p\bar \nabla_q \omega_{J_{-}}\right)_{J_{-}}-\omega_{J_{-}} (K\cdot, J\cdot)-\omega_{J_{-}}(J\cdot, K\cdot).
\end{equation}
Indeed this follows from a direct computation, similarly as in the proof of Proposition \ref{P3-k12}, by noting that $(\omega_J, J)$ is a compatible pair. To be more specific, applying Proposition \ref{P3-k12} to $(\omega_J, J)$, we have (see \eqref{E3-k22})
\[
(-d^*d\omega_J)_{J_{-}}=\frac{1}{2}\omega_J(K_2\cdot, J\cdot)+\frac{1}{2}\omega_J(J\cdot, K_2 \cdot)
\]
By compatibility of $(\omega_J, J)$, a direct computation gives
\[
\left(L_X\omega_J\right)_{J_{-}}=\frac{1}{2}\omega_J(L_XJ\cdot, J\cdot)+\frac{1}{2}\omega_J(J\cdot, L_XJ\cdot)
\]
Hence we have
\begin{equation}\label{E3-kj}
(-d^*d\omega_J+L_X\omega_{J})_{J_{-}}=\frac{1}{2}\omega_J(K\cdot, J\cdot)+\frac{1}{2}\omega_J(J\cdot, K\cdot). 
\end{equation}
We compute
\[
\frac{\p }{\p t} \omega_{J_{-}}= \left(\p_t\omega\right)_{J_{-}}-\frac{1}{2}\omega(K\cdot, J\cdot)-\frac{1}{2}\omega(J\cdot, K\cdot)
\]
Using the equation for $\p_t\omega$ and \eqref{E3-kj}, we have proved the claim. 
Given \eqref{E-c}, a direct maximum principle argument  then implies that if $\omega_{J_{-}}=0$ at $t=0$, it is preserved along the flow. Alternatively, one can also use \eqref{E-c} to compute directly that, in $[0, T]$, 
\[
\begin{split}
\frac{\p }{\p t}\int_M |\omega_{J_-}|^2 d\mu=&2\int_M \frac{\p \omega_{J_{-}}}{\p t}\omega_{J_-}+\int_M \omega_{J_-}*\omega_{J_-}\\
=&-2\int_M |\nabla \omega_{J_-}|^2+\int_M \nabla \omega_{J_-}*\omega_{J_-}+\int_M \omega_{J_-}*\omega_{J_-}\\
\leq& -\int_M |\nabla \omega_{J_-}|^2+C\int_M |\omega_{J_-}|^2,
\end{split}
\]
where $A*B$ denotes possible contractions of $A, B$ (we also use the fact that all metrics are smooth and hence have bounded geometry and are equivalent for $t\in [0, T]$). 
The statement then follows from the Gronwall's inequality.
\end{proof}

By Proposition \ref{P3-k12},  $K_1(\omega, J)=K_2(\omega, J)$ holds for a compatible pair, hence we get that

\begin{prop}Let $(\omega, J)$ be almost Hermitian structures which solves \eqref{Hflow4}. Then $(\omega, J)$ solves the system
\begin{equation}\label{E-Hflow5}
\begin{split}
\frac{\p \omega}{\p t}&=-d^*d\omega+L_X\omega\\
\frac{\p J}{\p t}&=K_1(\omega, J)+L_XJ
\end{split}
\end{equation}
where $X$ is the vector field defined in \eqref{E-x} and $K_1(\omega, J)$ is determined by
\begin{equation}\label{E-k2}
g(K_1x, y)=(-d^*d\omega)_{J_{-}}.
\end{equation}
\end{prop}

Now we can apply the DeTurck's trick and we prove Theorem \ref{T-3}  by obtaining  a smooth short time solution of \eqref{Hflow3}, for 
a compatible initial condition $(\omega_0, J_0)$. Indeed, we can first solve \eqref{Hflow4} with initial condition $(\omega_0, J_0)$ with a smooth solution for a short time. Proposition \ref{P-c}  then implies $(\omega(t), J(t))$ is an almost Hermitian structure, and hence it solves \eqref{E-Hflow5},  where $X$ is determined in \eqref{E-x}. Now let $\phi_t: M\rightarrow M$ is a diffeomorphism generated by $-X$ with $\phi_0=id$. Denote \[(\tilde \omega, \tilde J)=(\phi_t^*\omega(t), \phi_t^*J(t)).\] We can then compute \[
\begin{split}
&\frac{\p \tilde \omega}{\p t}=\phi_t^*(\p\omega/\p t+L_{-X}\omega)=-\phi_t^{*}(d^*d\omega)= -d^*d\tilde \omega\\
&\frac{\p \tilde J}{\p t}=\phi_t^*(\p J/\p t+L_{-X}J)=\phi_t^*(K_1(\omega, J))=K_1(\tilde \omega, \tilde J).
\end{split}
\] 
Hence $(\tilde \omega, \tilde J)$ solves \eqref{Hflow3}. \\

We emphasize that the uniqueness still holds, but  the usual proof of the uniqueness does not extend to this system directly, due to the fact that the equation for the involved diffeomorphisms is not parabolic. To be more specific,  for a diffeomorphism $\phi: M\rightarrow M$, denote \[\left(\phi^{-1}\right)^* X=X\left((\phi^{-1})^* \omega, (\phi^{-1})^* J, \bar g\right)\]
We want to consider  the operator
\[\phi\rightarrow  (-X) (\left(\phi^{-1}\right)^*\omega, \left(\phi^{-1}\right)^* J, \bar g).\]
If it were an elliptic operator, the proof for the uniqueness should follow the standard route; however this is not the case in our situation. Without loss of generality, we assume $\phi=id$ and $\delta \phi=Y$, then we have
$\delta \phi^{-1}=-Y$. We compute
\[\begin{split}
\delta\left(\phi^{-1}\right)^* (-X^k)=&\delta\left(- (\phi^{-1})^* \left(g^{pq}\bar \nabla_p (\omega_{qj}) \omega^{jk}\right)\right)\\
=&\delta\left(- (\phi^{-1})^* \left(g^{pq}\p_p (\omega_{qj}) \omega^{jk}\right)\right)\\
=&g^{pq}\p_p\left(L_Y\omega_{qj}\right)\omega^{jk}+\text{l.o.t.}\\
=&g^{pq}\p^2_{p, q}(Y^k)+g^{pq}\p^2_{p, j}(Y^l)\omega_{lq}\omega^{jk}+\text{l.o.t}
\end{split}\]
where $\text{l.o.t}$ stands for terms of $Y^k$ up to first derivative. We can choose a coordinate such that $g_{ij}=\delta_{ij}$ (at one point) and it follows that $\omega^{jk}=-\omega_{jk}$ (note that ${\omega^{jk}}$ is the inverse of $\omega$), hence the symbol of the linearized operator reads, for $\xi=(\xi_1, \cdots, \xi_{2n} )$, 
\[
\sigma L(\xi): Y^k\rightarrow Y^k |\xi|^2-\sum_j\xi_j \omega_{jk} \sum_p\xi_p\omega_{pl} Y^l
\]
We assume $|\xi|^2=1$ and denote $\eta_k=\sum_j\xi_j \omega_{jk}$, then $\sum \eta_k\eta_k=1$. Hence we can write the operator as $I-\eta^T\eta$; this operator is always nonnegative definite (and mostly positive indeed), but it has exactly one zero eigenvalue when $\eta_k=1$ for some $k$ ($\eta_i$ has to be $0$ for $i\neq k$). 
We shall also mention that the equation on $\phi$, \[\p_t\phi=-(-X) (\left(\phi^{-1}\right)^*\omega, \left(\phi^{-1}\right)^* J, \bar g)\] likely has a (unique) short time smooth solution, even it is only a degenerate parabolic system. This would be enough for the proof of uniqueness but we shall not pursue in this direction. Another way to get around might be that we try to choose $X$ slightly differently such that it still makes the system \eqref{Hflow4} parabolic and $(-X) (\left(\phi^{-1}\right)^*\omega, \left(\phi^{-1}\right)^* J, \bar g)$ is also parabolic on $\phi$, since we have some freedom on the choice of $X$. However we are not able to find such a choice of $X$.  Instead we will use Halmiton's Theorem \ref{thm-h}, as in Section 2, to give an alternative proof of existence and uniqueness as well.

\subsection{Uniqueness}
We consider the following system for a tamed pair $(\omega, J)$,
\begin{equation}\label{Hflow4-1}
\begin{split}
&\frac{\p \omega}{\p t}=-d^*d\omega\\
&\frac{\p J}{\p t}=K_2(\omega_J, J),
\end{split}
\end{equation}
where $K_2, d^*$ are both determined by the almost Hermitian structure $(\omega_J, J)$.

\begin{thm}\label{T-uniqueness}For a tamed pair $(\omega_0, J_0)$, there exists a unique smooth solution for a short time. When $(\omega_0, J_0)$ is compatible, then the compatibility condition is preserved along the flow and the solution coincides with \eqref{Hflow3}. In particular, a smooth solution of the system \eqref{Hflow3} with fixed initial condition is unique. 
\end{thm}

\begin{proof}We shall keep the discussion brief since this is similar to the situation in Theorem \ref{thm-tame}.  For a tamed pair $(\omega, J)$, the linearized operator of the system is not an elliptic, but a degenerate elliptic operator. However we can understand the degeneracy clearly, as in Theorem \ref{thm-tame}. First the leading term of the linearized operator of $-d^*d$ is still $-d^*d$ and an important feature is that it contributes zero to the second derivatives on $J$;  this operator has a null eigenspace, with its image contained in $\Ker (d^*)$, and $-d^*d$ is elliptic on $\text{Ker}(d^*)$. And $K_2(\omega_J, J)$ is  elliptic on $J$  (it does contributes second derivatives of $\omega$). Nevertheless, as in Theorem \ref{thm-tame}, we still have $L(\omega, J) (\theta, K)=(d^*\theta, 0)$ as the integrability condition of the system $E(\omega, J)=(-d^*d\omega, K_2)$ (clearly $L(\omega, J) E(\omega, J)=0$ and $L(\omega, J)$ is a (at most) degree one operator). The linearized operator  $DE$ is elliptic on the kernel of $\Ker(L)=\Ker(d^*)\otimes \cK$, where $\cK$ is the space of  all endomorphisms $K$ satisfying $JK+KJ=0$.  We can then apply Hamilton's Theorem \ref{thm-h} to conclude the existence and uniqueness of a smooth solution for a short time for this system.  

When $(\omega_0, J_0)$ is compatible,  there exists a short time smooth solution of \eqref{Hflow3}, by the results we proved in Section 3.2. Note that the compatibility condition is preserved and hence we have $K_2=K_1$. It then implies that  this solution also solves \eqref{Hflow4-1}. By uniqueness of \eqref{Hflow4-1} we then know that the flow preserves compatibility of $(\omega, J)$.  We also note that  any smooth solution of \eqref{Hflow3} of a compatible pair also solves \eqref{Hflow4-1}. By the uniqueness of \eqref{Hflow4-1} again, this implies the uniqueness of the system \eqref{Hflow3} as well. This completes the proof. 
\end{proof}

\subsection{Volume functional, special solutions and singularities}\label{CV}
Clearly a fixed point is $-d^*d\omega=0$, or equivalently, $d\omega=0$. Hence a fixed point is a symplectic structure. Another special  solution is $-d^*d\omega=\l \omega$ for some constant $\l$. Clearly $\l=0$ or $\l$ has to be negative. For negative $\l$, one can scale $\omega$ (correspondingly, the metric) to get any positive constant. For simplicity, we consider 
\begin{equation}\label{E3.4-1}d^*d\omega =\omega.\end{equation}
This is equivalent to  $d^*\omega=0$ and $\Delta \omega=\omega$. From this point of view, these special solutions of $d^*d$ flow are more rigid than Laplacian flow, since the special solutions of the latter are either symplectic or just satisfies $\Delta \omega=\omega$. We shall discuss in more details about this equation \eqref{E3.4-1}.

We shall note that there are many examples of almost Hermitian structure $(\omega, J)$ satisfy $d^*d\omega=\omega$. For example, up to a scaling, the nearly-K\"ahler structure on $S^6$ satisfies this  (indeed any strictly nearly Kahler example satisfies this, up to scaling).  Note that such an example gives a finite time singularity, as the case that positive Einstein metrics for the Ricci flow.
Indeed if $(\omega_0, J_0)$ is an almost Hermitian structure such that $d^*_0d\omega_0=\omega_0$. For any constant $\l>0$, $(\l\omega_0, J_0)$ defines another almost Hermitian structure such that $d^*_\l d=\l^{-1} d^*_0d$. Hence $((1-t)\omega_0, J_0)$ is the unique solution for the system and it collapses at $t=1$. 

It would be highly preferable to understand general singularities of this system. The discussions and results, as in other geometric evolution equations, in particular the Ricci flow, will certainly be enlightening. We shall consider this elsewhere. 

We also note that the volume functional decreases along such a flow. For the $d^*d$ flow (or Laplacian flow), one can compute directly that
\[
\frac{d}{dt} \int_M \frac{\omega^n}{n!}=\int_M -d^*d\omega \wedge \frac{\omega^{n-1}}{(n-1)!}=\int_M -d^*d\omega \wedge *\omega =-\int_M (d\omega, d\omega) \frac{\omega^n}{ n!}.
\]
Hence the volume functional is decreasing along the flow.
\begin{prop}The volume functional is decreasing along the $d^*d$ flow. \end{prop}

Hence suppose the flow exists for all time and converges smoothly, then the limit is a symplectic structure. Since a symplectic structure might not exist  on a compact almost complex manifold, the singularities are inevitable in general. And it would be an interesting problem to study the formation of singularities. 

\subsection{An example: warped product}
First we consider a simple example of $d^*d$-flow on $\R^4$, with coordinate $(x, y, z, w)$. Let $J_0$ be the standard complex structure satisfying \[J_0\p_x=\p_y, J_0\p_y=-\p_x, J_0\p_z=\p_w, J_0\p_w=-\p_z.\] 
We consider  a nondengenerate two form, with $a=a(z, w), b=b(x, y)$ both positive, \[\omega_0=a(z, w)dx\wedge dy+b(x, y)dz\wedge dw.\] Clearly $(\omega_0, J_0)$ defines a Hermitian structure on $\R^4$ and $d\omega_0\neq 0$ unless both $a, b$ are constant. First we assume $a=1$ and $b$ is periodic in $x$ and $y$, and hence $\omega_0$ descends to a nondegenerate two form on $T^4=\R^4/\Z^4$, a compact tori. We have the following,
\begin{prop}For a Hermitian structure $(\omega_0, J_0)$ on $T^4$ with $\omega_0=dx\wedge dy+b_0(x, y) dz\wedge dw$, there exists a unique long time solution to $d^*d\omega$ flow, which is of the form 
$\omega_t=dx\wedge dy+b(t, x, y) dz\wedge dw, J_t=J_0$ such that 
\begin{equation}\label{E-b}
\p_t b=\Delta_{x, y} b-b^{-1} |\nabla_{x, y} b|^2.
\end{equation}
In particular, when $t\rightarrow \infty$, $\lim b$ is a positive constant, hence $\omega_t$ converges to a symplectic form on $T^4$. 
\end{prop}

\begin{proof}By the uniqueness and existence, \eqref{E-b} follows from a routine computation. To be more precise, for any nondegenerate form $\omega=dx\wedge dy+b(x, y) \wedge dw$ (we assume $b>0$ as always), it defines a Hermitian structure with $J_0$. Moreover, 
\[
dx, dy, \sqrt{b} dz, \sqrt{b} dw
\]
gives an orthonormal coframe. Hence a direct  computation gives, 
\[
\begin{split}
d\omega=&\p_x b  dx\wedge dz\wedge dw+\p_y b dy\wedge dz\wedge dw \\
*d\omega=&-\frac{\p_x b}{b} dy+\frac{\p_y b}{b} dx\\
d*d\omega=&-\p_x\left(\frac{\p_x b} {b}\right)dx\wedge dy-\p_y \left(\frac{\p_y b}{ b}\right) dx\wedge dy\\
-*d*d\omega=& b\left(\p_x\left(\frac{\p_x b} {b}\right)+\p_y \left(\frac{\p_y b}{ b}\right)\right) dz\wedge dw
\end{split}
\]
It follows that 
\begin{equation}\label{E-b01}-d^*d\omega = \left(\p_x^2+\p_y^2\right) b-b^{-1} \left(b_x^2+b_y^2\right)=\Delta b-b^{-1}|\nabla b|^2,\end{equation}
where $\Delta, \nabla$ take the obvious meaning, using the flat metric on $T^2$ and coordinates $x, y$. Now we consider the equation on $T^2$
\begin{equation}\label{E-b1}
\p_t b=\Delta b-b^{-1}|\nabla b|^2,
\end{equation}
with the initial condition $b(0)=b_0$. The equation \eqref{E-b1} mimics the Harmonic flow equation in some way but the nonlinear term $-b^{-1}|\nabla b|^2$ has the sign which is towards our favorite. Indeed if we consider the equation for $\log b$, we get the standard heat equation,
\[
\p_t \log b=\Delta (\log b).
\]
From this we see that the equation has a long time unique solution and $\log b$ converges to a constant on $T^2$ when time goes to infinity. 
Now define $\omega_t=dx\wedge dy+b(t, x, y)dz\wedge dw$ on $T^4=T^2\times T^2$. We see that $(\omega_t, J_0)$ solves the $d^*d$-flow for a compatible pair; indeed since $d^*d\omega_t$ is always $J_0$ compatible, $J_0$ stays fixed and the evolution equation for $\omega_t$ is then evident, by \eqref{E-b01}. By the uniqueness, $d^*d$-flow is reduced to \eqref{E-b} in this case. This completes the proof. 
\end{proof}

When $a$ is not a constant, one can compute that, for $\omega_0=a(z, w) dx\wedge dy+b(x, y) dz\wedge dw$,
\[
-d^*d\omega=\frac{b}{a}\left(\p_x\left(\frac{\p_x b} {b}\right)+\p_y \left(\frac{\p_y b}{ b}\right)\right)+\frac{a}{b} \left(\p_z\left(\frac{\p_z a} {a}\right)+\p_w \left(\frac{\p_w a}{ a}\right)\right).
\]
The system becomes complicated since we cannot separate variables anymore and $J$ will also be evolved along the flow.

We can generalize this example to more general cases such as  ``warped product" and fiber bundles. For simplicity we only consider warped product here. Let $(M_1, \omega_1, J_1)$ and $(M_2, \omega_2, J_2)$ be two almost K\"ahler manifolds. Consider an almost Hermitian structure on \[M=M_1\times M_2, \omega_f(x, y)=\omega_1(x)+f(x)\omega_2(y), J,\] where $f: M_1\rightarrow \R_{+}$ is a positive smooth function on $M_1$ and $J=J_1\oplus J_2$. 

\begin{thm}For any almost Hermitian structure $(M, \omega_f, J)$, the $d^*d$-flow exists for all time and it converges to a almost K\"ahler structure with symplectic form $\omega_\infty=\omega_1+a\omega_2$, for some positive constant $a$.
\end{thm}

\begin{proof}The proof is pretty similar to the example we discussed above. It is a straightforward computation to show that, 
\[
-d^*d (\omega_f)=\left(\Delta_1 f+(n-2)f^{-1}|\nabla f|^2\right) \omega_2,
\]
where $\Delta_1, \nabla_1$ are the Laplacian operator and covariant derivative on $M_1$ respectively, and $n$ is the half dimension of $M_2$. Hence, if $f(t)$ is a time-dependent function on $M_1$ satisfying the equation,
\[
\p_t f=\Delta_1 f+(n-2)f^{-1}|\nabla_1 f|^2, 
\]
then  $(\omega_f, J)$ solves the $d^*d$-flow (by the uniqueness). When $n=1$, we have
\[
\p_t \log (f)=\Delta_1 \log (f).
\]
When $n\geq 2$, then we have
\[
\p_t \left(f^{n-1}\right)=\Delta_1 \left(f^{n-1}\right).
\]
It follows that the flow exists for all time and $f$ converges to a positive constant, by the standard theory of the heat equation on a   compact Riemannian manifold. Hence the limit is a symplectic structure of the form $\omega_\infty=\omega_1+a \omega_2$ for some positive constant $a$. 
\end{proof}

\subsection{Extension and evolution equations}
We prove an extension theorem for the flow in this section, 

\begin{thm}\label{T-5}Suppose $(\omega, J)$ is a smooth solution of the $d^*d$-flow in $[0, T]$ and suppose the Riemannian curvature $Rm$, $|\nabla \omega|$ and $|\nabla^2\omega|$ are uniformly bounded, then all high derivatives of $\omega, J$ are also uniformly bounded and hence the flow can be extended across $T$ (in a uniform way). In other words, if $[0, T_0)$ is the maximal interval with the smooth solution for $d^*d$-flow and $T_0<\infty$, then
\[
\limsup \left(|Rm|+|\nabla \omega|+|\nabla^2\omega|\right)\rightarrow \infty, \; \mbox{when}\; t\rightarrow T_0.
\]
\end{thm}

A standard way to prove this type of theorem is to consider evolution equation of various geometric quantities such as curvature. We write $\theta_{ij}=-d^*d\omega_{ij}$. 
We need an identity for $d\alpha$, where $\alpha$ is a $p$-form
\[
d\alpha(X_0, X_1, \cdots, X_p)=\sum_{i=0}^p (-1)^i\nabla_{X_i} \alpha(X_0, \cdots, \hat X_i, \cdots, X_p).
\]
Hence we can write
\begin{equation}
(d\omega)_{ijk}=\frac{1}{3}\left(\nabla_i \omega_{jk}-\nabla_j\omega_{ik}+\nabla_k\omega_{ij}\right)
\end{equation}
It follows that
\begin{equation}
-d^*d\omega_{ij}=\theta_{ij}=g^{pq}\left(\nabla_p\nabla_q \omega_{ij}-\nabla_p\nabla_i\omega_{qj}+\nabla_p\nabla_j\omega_{qi}\right)
\end{equation}
Recall that we also have 
\[
\theta_{ij}=g^{pq}\left(\p^2_{p, q}\omega_{ij}-\p^2_{p, i}\omega_{qj}+\p^2_{p, j}\omega_{qi}\right)+O(1).
\]
It is also a direct computation to check that for any two-form $\phi$,
\begin{equation}
\nabla_p\nabla_q \phi_{ij}-\nabla_p\nabla_i\phi_{qj}+\nabla_p\nabla_j\phi_{qi}=\p^2_{p, q}\phi_{ij}-\p^2_{p, i}\phi_{qj}+\p^2_{p, j}\phi_{qi}+O(1)
\end{equation}
We continue to compute, 
\[
\frac{\p J^j_i}{\p t}=K^j_i=\frac{1}{2}(\theta_{ip}-\theta_{kl}J^k_iJ^l_p)g^{jp}
\]
We can then compute
\begin{equation}
\frac{\p g_{ij}}{\p t}=\frac{1}{2}\left(\theta_{ik}J^k_j+\theta_{jk}J^k_i\right):=h_{ij}. 
\end{equation}
 We can then compute,
\begin{equation}\label{E3.5-1}
\begin{split}
h_{ij}=&\frac{1}{2}g^{pq}\left(\p^2_{p,q} \omega_{ik}-\p^2_{p,i}\omega_{qk}+\p^2_{p,k}\omega_{qi}\right)J^k_j\\
&+\frac{1}{2}g^{pq}\left(\p^2_{p, q}\omega_{jk}-\p^2_{p, j}\omega_{qk}+\p^2_{p, k}\omega_{qj}\right)J^k_i+O(1)\\
=&g^{pq}\p^2_{p, q} g_{ij}-\frac{1}{2}g^{pq} \left(\omega_{ik}\p^2_{p, q} J^k_j+\omega_{jk}\p^2_{p, q} J^k_i\right)\\
&+\frac{1}{2}g^{pq}\left(J^k_j\p^2_{p, k}\omega_{qi}+J^k_i\p^2_{p, k}\omega_{qj}-J^k_j\p^2_{p,i}\omega_{qk}-J^k_i \p^2_{p, j}\omega_{qk}\right)\\
=&g^{pq}\p^2_{p, q} g_{ij}+\tilde h_{ij}
\end{split}
\end{equation}

Given the variation of the metric, one can compute directly the variation of connections and curvatures. These formulas are well-known, for example, see \cite{CLN}. First we have the evolution equation for the connections, 
\begin{prop}\label{P3-connection}The variation of the connection is given by,
\begin{equation}\label{E3.5-2}
\frac{\p }{\p t}\Gamma^k_{ij}=\frac{1}{2} g^{kl}\left(\nabla_i h_{lj}+\nabla_j h_{li}-\nabla_l h_{ij}\right)
\end{equation}
\end{prop}

To compute the evolution equation of the curvature, we can use the following expression,
\[
\begin{split}
R_{ijkl}=&\left(\p_i\Gamma^p_{jk}-\p_j \Gamma^p_{ik}\right)g_{pl}+\Gamma*\Gamma\\
=&\frac{1}{2}\left(\p^2_{i,k}g_{jl}+\p^2_{j, l}g_{ik}-\p^2_{i,l}g_{jk}-\p^2_{j, k}g_{il}\right)+\Gamma*\Gamma,
\end{split}
\]
where $\Gamma*\Gamma$ denotes two quadratic terms of connections (with possible contractions by the metric $g$). We have,
\begin{prop}
\begin{equation}\label{E3.5-c}
\begin{split}
\p_t R_{ijkl}=&\frac{1}{2}\left(\nabla_i\nabla_k h_{jl}+\nabla_j\nabla_l h_{ik}-\nabla_j\nabla_k h_{il}-\nabla_i\nabla_l h_{jk}\right)\\
=&\frac{1}{2}\left(\p^2_{i,k}h_{jl}+\p^2_{j, l}h_{ik}-\p^2_{i,l}h_{jk}-\p^2_{j, k}h_{il}\right)+\mbox{l.o.t}\\
=&D R_{ijkl}+\frac{1}{2}\left(\p^2_{i,k}\tilde h_{jl}+\p^2_{j, l}\tilde h_{ik}-\p^2_{i,l}\tilde h_{jk}-\p^2_{j, k}\tilde h_{il}\right)+\mbox{l.o.t},
\end{split}
\end{equation}
where denote the operator $D$ to be
\[
D=g^{pq}\nabla_p\nabla_q=-\nabla^*\nabla.
\]
\end{prop}

We can also compute the evolution equation of $\nabla \omega$. Note that $\nabla \omega, \nabla J$ are mutually determined each other; indeed
\[
\nabla_k(J^j_i)=g^{jl}J_{a}^cJ_l^b\nabla_i(\omega_{bc})
\]
$\nabla \omega=\nabla J=0$ is equivalent to the fact that $(\omega, J, g)$ is a Kahler structure. 
\begin{prop}We compute
\begin{equation}\label{E3.5-3}
\begin{split}
\frac{\p}{\p t}\nabla_k\omega_{ij}=&\nabla_k\frac{\p}{\p t} \omega_{ij}-(\p_t \Gamma^m_{ki}) \omega_{mj}-(\p_t\Gamma^m_{kj})\omega_{im}\\
&=\nabla_k(\theta_{ij})-(\p_t \Gamma^m_{ki}) \omega_{mj}-(\p_t\Gamma^m_{kj})\omega_{im}.
\end{split}
\end{equation}
\end{prop}

For a geometric quantity $Q$, we would like to write the evolution equation of the form 
\[
\p_t Q=D Q+r(\omega, J),
\]
It turns out that $r(\omega, J)$ consists of not only lower order terms; for example \eqref{E3.5-1}, \eqref{E3.5-3} and \eqref{E3.5-c} seem to be rather complicated.
We can also consider evolution equation of various forms such as $d\omega, d^*d\omega$;  such quantities behave nicely along the flow. For example
we compute, 
\[
\frac{\p }{\p t} d\omega=- d d^* d \omega=-\Delta d\omega. 
\] 
And we also have
\[
\frac{\p}{\p t}d^*d\omega=-\Delta d^*d\omega+\mbox{l.o.t}
\]
We recall the Bochner-Weinzebock formula.  For any two-form $\psi$,
\[
\Delta \psi_{ij}=\nabla^*\nabla \psi_{ij}+\left(R_{i}^p\psi_{pj}+R_{j}^p\psi_{ip}\right)+g^{kl}R_{ijk}^p\psi_{lp}
\]
For any $p$-form $\psi$, we have the following, \[
\Delta\psi=-D \psi+Rm*\psi.
\]
In particular, we can write, for example, 
\[
\frac{\p}{\p t}d^*d\omega=D d^*d\omega+Rm*d^*d\omega+\mbox{l.o.t}
\]
Nevertheless,  it does not seems to be straightforward to establish an extension and/or improved regularity result for the $d^*d$ flow in the usual fashion. Instead we apply the standard regularity theory for parabolic systems to the modified flow. 

\begin{proof}Let $(\omega, J)$ be the smooth solution of the flow in $[0, T]$ with $|\nabla \omega|, |\nabla^2\omega|, |Rm|$ all bounded. Consider the equivalent system for $(\tilde\omega, \tilde J)$,
\begin{equation}\label{Hflow3.5-1}
\begin{split}
&\frac{\p \tilde \omega}{\p t}=-d^*d\tilde \omega+L_X\tilde\omega\\
&\frac{\p \tilde J}{\p t}=K_2(\tilde\omega, \tilde J)+L_X\tilde J,
\end{split}
\end{equation}
where $X=X(\tilde \omega, \tilde J)$. By the uniqueness and existence of the flow, we know that this system also has a unique smooth solution in $[0, T]$, such that $\omega=\phi^*\tilde \omega, J=\phi^*\tilde J$, where $\phi$ is the diffeomorphism generated by $X$. Hence we only need to prove the extension for \eqref{Hflow3.5-1}. By the results in Section 3.3, the system has the structure as follows,
\[
\begin{split}
&\frac{\p \tilde \omega}{\p t}=\tilde g^{pq}\p^2_{p,q}\tilde \omega+O(\p\tilde \omega, \p \tilde J)\\
&\frac{\p \tilde J}{\p t}=\frac{1}{2}\tilde g^{pq}\left(\p^2_{p, q}\tilde J^j_i-\p^2_{p, q}\tilde J^b_a \tilde g_{ib} \tilde g^{aj}\right)+O(\p^2\tilde\omega, \p\tilde \omega, \p \tilde J).
\end{split}
\]
First of all, if $|\nabla \omega|, |\nabla^2\omega|, |Rm|$ are all bounded, then all the metrics are uniformly equivalent in $C^{1, \alpha}$ norm. Applying the maximal regularity theory to the first system for $\tilde \omega$ (we know all coefficients are uniformly in $C^\alpha$), this follows that $\omega$ is in $C^{2, \alpha}$. Now we consider the second system for $\tilde J$, which is parabolic in $J$. Since $\omega$ is in $C^{2, \alpha}$, the lower order term $O(\p^2\tilde\omega, \p\tilde \omega, \p \tilde J)$ is also in $C^\alpha$. The key is that this term does not involve second derivatives of $\tilde J$. Hence the maximal regularity implies that $\tilde J$ is in $C^{2, \alpha}$. Then by the standard boot-strapping argument we know that $\tilde \omega, \tilde J$ are uniformly bounded in $C^{k, \alpha}$ for any $k$. This allows us to extend the flow across time $T$, in a uniform way. This completes the proof. 
\end{proof}

\begin{rmk}The paper was written in 2014 and the results were presented in University of British Columbia and Stony Brook University. We have seen the recent work of J. Lotay and Y. Wei \cite{LW} on the Lalacian flow for closed $G_2$ structure. In particular the authors proved a Shi's type estimate and this will give the desired estimates and extension. Our flow shares some similarities with Laplacian flow introduced by R. Bryant on closed $G_2$ structure \cite{BX}. A similar strategy as in Lotay-Wei should give a Shi's type estimate for $d^*d$-flow and $d^*d$-Ricci flow. This will be discussed in a separate paper. 
\end{rmk}

\subsection{Existence of a symplectic structure}

In this section, we propose a conjecture on the existence of a symplectic structure on a compact  almost Hermitian manifold when dimension is six or above. Let $(M, \omega, J)$ be such an almost Hermitian structure. If $d\omega=0$, then $d^*\omega=0$ since $*\omega=\omega^{n-1}/(n-1)!$. Note that when $n=2$ (hence $M$ is of dimension 4), $d\omega=0$ is equivalent to $d^*\omega=0$. However, when the dimension is six or above, $d^*\omega=0$ is strictly weaker. (In the Hermitian setting, such a structure is called a ``balance" form and it has been studied extensively). Nevertheless, 
first we require $d(\omega^{n-1})=0$. Hence $\omega^{n-1}$ defines a cohomology class in $H^{2n-2}(M, \R)$ but it can be zero class (even though $\omega$ is nondegenerate). An example is the round sphere of dimension six, which satisfies $d(\omega^2)=0$. The second requirement is that the cohomology class $[\omega^{n-1}]$ is not zero. We consider the cohomology class $[\alpha]$ in $H^2(M, \R)$, which is the Poincare dual of $[\omega^{n-1}]/(n-1)!$ by the duality induced by the Hodge star. Hence we need $[\alpha]^n>0$. In the end we require $[\alpha]^{n-1}=[\omega^{n-1}]$. 
Clearly, if $d\omega=0$, these are all obvious necessary conditions. The only difference with the obstructions we have discussed in the introduction is that we specify these conditions in a more geometric way, instead of in purely topological datum ( note that the almost Hermitian structure and the Hodge $*$ are only auxiliary since any such structure compatible with $\omega$ will be good enough). Hence we propose the following, 

\begin{q}\label{q-s}
Let $(M, \omega, J)$ be an almost Hermitian structure such that $d(\omega^{n-1})=0$. Denote $[\alpha]$ to be the cohomology class in $H^2(M, \R)$ such that 
\[
[\alpha]=*[\omega^{n-1}]/(n-1)!.
\]
We assume further that $[\alpha]^n>0$ and $[\alpha]^{n-1}=[\omega^{n-1}]$. Is there a symplectic structure in $[\alpha]$?
\end{q}

Even though $d(\omega^{n-1})=0$ is still a nontrivial condition, we hope Problem \ref{q-s} can give some approach in geometric analysis to the existence of symplectic forms. A trouble using the $d^*d$-flow to approach this problem is that $d(\omega^{n-1})=0$ seems not to be a condition which is preserved along the flow.

\subsection{Laplacian flow}

Similarly we can also consider the Laplacian flow for a compatible pair $(\omega_0, J_0)$, 

\begin{equation}\label{Lflow-1}
\begin{split}
&\frac{\p \omega}{\p t}=-\Delta \omega\\
&\frac{\p J}{\p t}=K,
\end{split}
\end{equation}
where $K$ is given by
\begin{equation}
g(Kx, y)=(-\Delta \omega)_{J_{-}}.
\end{equation}

However, due to the complexity of the linearized operator of $\Delta$, we do not succeed proving the short time existence at the moment. We can ask the following, 

\begin{q} Does \eqref{Lflow-1} have a smooth short time solution starting from an almost Hermitian structure? 
\end{q}

\section{$d^*d$-Ricci flow}
As in last section, for a variation $(\theta, K)$ of a compatible pair $(\omega, J)$,  the choice of $K$ is not unique even with $\theta$ fixed. When we deform $\omega$ by $-d^*d\omega$, there are still fruitful choices for $K$. A particularly interesting  example is that $B=B(\omega, J)$ is determined by the Ricci curvature (note that $g(Bx, y)$ should be symmetric and anti-$J$ invariant),
\[
g(Bx, y)=Ric(Jx, y)+Ric(x, Jy).
\]
Hence we study the $d^*d$-Ricci flow for a compatible pair $(\omega, J)$, 
\begin{equation}\label{HRflow}
\begin{split}
&\frac{\p \omega}{\p t}+d^*d\omega=0\\
&\frac{\p J}{\p t}=K_1(\omega, J)+B(\omega, J),
\end{split}
\end{equation}
where $K_1$ is as in \eqref{E3-k1}. First we note the following simple fact, 
\begin{prop}Suppose there exists a smooth solution $(\omega, J)$ of the $d^*d$-Ricci flow and the initial data is symplectic, $d\omega_0=0$. Then $\omega=\omega_0$ and the flow is reduced to 
\[
\frac{\p J}{\p t}=B(\omega_0, J).
\]
\end{prop}
\begin{proof}The proof is straightforward. Suppose $d\omega_0=0$, we claim that $d\omega=0$ along the flow. Hence $\p_t \omega=0$. This completes the proof if we establish the claim. Indeed we have
\[
\p_t \int_M |d\omega|^2 dv= \int_M (d\omega, -dd^*d\omega)dv+|d\omega|^2 \p_t dv+\p_t g *d\omega*d\omega dv\leq C\int_M|d\omega|^2 dv.
\]
By Gronwall's inequality, we know that $\int_M |d\omega|^2dv=0$. 
\end{proof}
Hence if $\omega_0$ is a symplectic form ($d\omega_0=0$), the flow is then reduced to the anti-complexied Ricci flow, as studied in \cite{LW}. 

\begin{rmk}The anti-complexied Ricci flow is the gradient flow of the functional of compatible almost complex structures on a compact symplectic manifold $(M, \omega)$,
\[
\int_M |\nabla J|^2 dv
\]
The proof of its short time existence presented in \cite{LW} is already technical, see comments in \cite{ST2} for example; we  thank Prof. G. Tian for bringing the complexied-Ricci  flow to our attention in 2007. Our motivation is  different since our main goal is to deform a non-degenerate two form to a symplectic form. Nevertheless, the anti-complexied Ricci flow is a special case of $d^*d\omega$-Ricci flow. We should mention that one can also consider Laplacian-Ricci flow. 
\end{rmk}

\subsection{Short time existence}

First we prove the short time existence and uniqueness of the $d^*d\omega$-Ricci flow. An advantage is that our flow is invariant under the diffeomorphism group, compared with the anti-complexied Ricci flow. 

\begin{thm}For any initial almost Hermitian structure $(\omega_0, J_0)$, there exists a unique smooth solution of the $d^*d$-Ricci flow for a short time. 
\end{thm}

\begin{proof}The proof is essentially the same as in $d^*d$-flow. We denote the principle symbol of $B(\omega_J, J)$ on $J$ by $Q$ and we will show that $Q$ has only nonnegative eigenvalues in below. All other arguments are the same and  we shall keep it brief. 
We consider a more general system for a tamed pair $(\omega, J)$, 
\begin{equation}\label{HRflow-1}
\begin{split}
&\frac{\p \omega}{\p t}+d^*d\omega_J=L_X\omega_J+g^{pq}\bar \nabla_p\bar \nabla_q (\omega_{J_{-}})\\
&\frac{\p J}{\p t}=K_2(\omega_J, J)+B(\omega_J, J)+L_XJ.
\end{split}
\end{equation}
The notion is the same as in the last section, where $\bar \nabla$ is the covariant derivative of a fixed background metric $\bar g$, and $g$ is determined by the pair $(\omega_J, J)$, as well as $K_2$ and $X$. By the results in the last section, the principle symbol of this system reads
\[
\begin{pmatrix} I& 0\\
*& I+Q
\end{pmatrix}
\]
which is in particular elliptic (parabolic), and hence there exists a unique smooth solution for a short time. Next we want to show that if $(\omega_0, J_0)$ is compatible, it remains so along the flow. The argument is similar to Proposition \ref{P-c}. Indeed, we have $\omega_J(Bx, Jy)+\omega_J(Jx, By)=0$, it then follows the same argument, we have
\[
\p_t \omega_{J_{-}}=\left(g^{pq}\bar \nabla_p\bar \nabla_q \omega_{J_{-}}\right)_{J_{-}}-\omega_{J_{-}}(K\cdot, J\cdot)-\omega_{J_{-}}(J\cdot, K\cdot),
\]
where $K=K_2+L_XJ+B(\omega_J, J)$. A standard argument using Gronwall's inequality or the maximum principle implies that the compatible condition is preserved.
Now by a gauge transformation, we prove that there exists a smooth solution of the system (we assume $(\omega_0, J_0)$ is compatible) 
\begin{equation}\label{HRflow-1a}
\begin{split}
&\frac{\p \omega}{\p t}+d^*d\omega=0\\
&\frac{\p J}{\p t}=K_2(\omega, J)+B(\omega, J).
\end{split}
\end{equation}

Regarding the uniqueness, we consider the system for the tamed pair and apply Hamilton's theorem (see Theorem \ref{thm-h})
\begin{equation}\label{HRflow-2}
\begin{split}
&\frac{\p \omega}{\p t}+d^*d\omega_J=0\\
&\frac{\p J}{\p t}=K_2(\omega_J, J)+B(\omega_J, J).
\end{split}
\end{equation}
The integrability condition is still $L(\omega, J)(\theta, K)=(d^*\theta, K)$. Since $K_2+B$ is elliptic on $J$, then the linearized operator of the system is elliptic on $\Ker(d^*)\oplus \cK$. Hence the system has a unique smooth solution for a short time. Now if the initial data is compatible, \eqref{HRflow-1a} has a compatible solution $(\omega, J)$, which is clearly also a solution of \eqref{HRflow-2}. Hence we also show that the compatible condition is preserved by \eqref{HRflow-2}. By the uniqueness of \eqref{HRflow-2} again, this proves the uniqueness for \eqref{HRflow-1a}. This completes the proof, provided the following proposition.
\end{proof}

\begin{prop}The principle symbol $Q$ of $B$ on $J$ is nonnegative definite and $J\rightarrow K_2(\omega, J)+B(\omega, J)$ is an elliptic operator.
\end{prop}

\begin{proof}This is essentially proved in \cite{LW} (see Proposition 3.12). The proof is straightforward and relies on the direct computation.
We write
\[
B^k_j=g^{ka}\left(J^b_aR_{bj}+J^b_jR_{ba}\right).
\]
Note that we also have
\[
B=\omega^{-1}(Ric-Ric(J\cdot, J\cdot))
\]
The leading order of the Ricci curvature reads
\[
R_{jk}=\frac{1}{2}g^{pq}\left(\p^2_{q, j}g_{kp}+\p^2_{q, k}g_{jp}-\p^2_{p, q}g_{jk}-\p^2_{j, k}g_{pq}\right).
\]
Now suppose $\delta J^k_j=K^k_j$, and denote $h_{ij}=\omega_{ik}K^k_j$, then the principle symbol of $DB$ on $K^k_j$ behaves the same as the principle symbol on $h_{ij}$
and we compute the principle symbol of $DB$ as
\[
\begin{split}
[\sigma DB](\xi)(h)=\frac{1}{2} \left(\omega^{-1}\right)^{ki} (\delta_i^a\delta_j^b-J^a_iJ^b_j) g^{pq} \left(\xi_q\xi_b h_{ap}+\xi_p\xi_a h_{bq}-\xi_p\xi_qh_{ab}-\xi_a\xi_bh_{pq}\right)
\end{split}
\]
From here the discussion is then exactly the same as in \cite{LW} (see Section 3.6 and 3.9). 
This operator has zero eigenvalues, for example, $h_{ab}=(\delta_i^a\delta_j^b-J^a_iJ^b_j)(\xi_i\xi_j)$ and $h_{ab}=J^p_b\xi_a \xi_p+J^p_a\xi_b\xi_p$ are both zero eigenvectors.
Moreover, all other nonzero eigenvalues are positive. It completes the proof. 
\end{proof}

Similar as in Theorem \ref{T-5}, we have the following extension result,

\begin{thm}Suppose $(\omega, J)$ is a smooth solution of the $d^*d$-Ricci flow in $[0, T]$ and suppose the Riemannian curvature $Rm$, $|\nabla \omega|$ and $|\nabla^2\omega|$ are uniformly bounded, then all high derivatives of $\omega, J$ are also uniformly bounded and hence the flow can be extended across $T$ (in a uniform way). In other words, if $[0, T_0)$ is the maximal interval with the smooth solution for $d^*d$-Ricci flow and $T_0<\infty$, then
\[
\limsup \left(|Rm|+|\nabla \omega|+|\nabla^2\omega|\right)\rightarrow \infty, \; \mbox{when}\; t\rightarrow T_0.
\]
\end{thm}

\subsection{Special solutions and nearly Kahler manifold}The special solutions of the $d^*d$-Ricci flow read 
\begin{equation}\label{E-s1}
d^*d\omega=\l\omega, \; Ric(\cdot, \cdot)=Ric(J\cdot, J\cdot).
\end{equation}
for some constant $\l$ (clearly, $\l$ has to be nonnegative). When $\l=0$, $d\omega=0$ and hence this is a symplectic manifold with a compatible almost Hermitian metric which has $J$-invariant Ricci tensor. Clearly any Kahler structure satisfies this condition. Indeed, Blair and Ianus \cite{BI} proposed to study such an almost Kahler structure with $J$-invariant Ricci tensor and asked whether this is a sufficient condition for $J$ to be integrable. When the dimension is six or above, there are compact examples of almost Kahler structure with $J$-invariant Ricci tensor which are not Kahler, as first constructed in \cite{DM}; this problem remains open when the dimension is four, and it is related to the so-called Goldberg conjecture which is studied extensively in literature. We refer to, for example, \cite{AD} and reference in for more discussions. When $\l>0$, first we note that the dimension has to be six or above. Indeed we have $d^*\omega=0$, this would force $d\omega=0$ if the dimension is four, a contradiction. In particular, in this case we have
\[
\Delta \omega=(d^*d+dd^*)\omega=\l \omega.
\]
Hence $\omega$ is an eigenform of Hodge-Laplacian $\Delta$ with positive eigenvalue. 
Note that $d^*d\omega$ is scaling invariant for $\omega\rightarrow k\omega$, we can then ask $\l$ to be any positive constant. 
The equation \eqref{E-s1} is of particular interest in dimension 6 when $\l\neq 0$. A (strictly) \emph{nearly Kahler manifold}, first introduced by A. Gray,  is a very special almost Hermitian structure, which is in particular Einstein manifold with positive scalar curvature. We refer to \cite{Bryant, Verbtisky} for example for references. We only recall that a stricly nearly K\"ahler manifold can be defined to be a compact almost Hermitian manifold of real dimension six such that there exists a normalized $(3, 0)$ form $\Omega$ ($|\Omega|=1$) such that
\begin{equation}\label{E-nk}
d\omega=3 \mu Re(\Omega), d(Im \Omega)=-2\mu \omega^2,
\end{equation}
for some real nonzero constant $\mu$.  

\begin{prop}A strictly nearly Kahler structure satisfies \eqref{E-s1} with a suitable choice of $\mu$ and $\l$. 
\end{prop}

\begin{proof}The proof is straightforward. Suppose $(M, \omega, J, \Omega)$ is a strictly nearly Kahler manifold. 
Since $d\omega$ is the real part of a $(3, 0)$ form, we have, $\omega\wedge d\omega=0$, namely, $d^*\omega=0$. The metric is Einstein, so it has in particular $J$-invariant Ricci curvature.
Now we  show $d^*d\omega=\l \omega$.
Let $*$ be the Hodge star, extended complex linearly to complex forms. Since $\Omega$ is a $(3, 0)$-form, we have
\[
*\Omega=-\i \Omega,\; *\bar{\Omega}=\i \bar{\Omega}. 
\]
It follows that, using \eqref{E-nk}
\[
* d\omega=3\mu *Re(\Omega)=\frac{3\i}{2} \mu(\bar\Omega-\Omega)=3\mu Im(\Omega)
\]
We compute
\[
 d^*d\omega=-*d*d\omega=-3\mu *dIm(\Omega)=6\mu^2 *\omega^2=12\mu^2\omega.
\]
By taking $\l=12\mu^2$, we complete the proof. 
\end{proof}

We can ask the following, 

\begin{q} Let $(M, \omega, J)$ be an almost Hermitian compact manifold of dimension 6, if it satisfies \eqref{E-s1}, is it a strictly nearly K\"ahler structure?
\end{q}

Note that $d^*d\omega=\l \omega (\l\neq 0)$ is a very strong restriction. It is equivalent to the following, $d^*\omega=0$ and $\Delta \omega=\l \omega$. For an almost Hermitian manifold, there is a decomposition $d=d^{2, -1}+\p+\bar \p +d^{-1, 2}$ (by Leibnitz rule on forms), where $d^{2, -1}$ is the Nijenhaus tensor, and $\p$ is the $(1, 0)$
part of $d$. Note that $d^*\omega=0$ is equivalent to $\omega\wedge d\omega=0$ (in dimension six); in other words, $\omega\wedge \p \omega=0$, or $\p (\omega^2)=0$.  
Note that $d\omega$ is the real part of a $(3, 0)$ form if and only if $\p \omega=0$. We can give the following description of a (strictly) nearly Kahler structure as follows.

\begin{prop}Suppose $(\omega, J)$ is an almost Hermitian (compact) manifold in the dimension 6. Then $(\omega, J)$ is a nearly Kahler structure if and only if $\p \omega=0$ and $d^*d\omega=\l \omega$ for some positive constant $\l$. 
\end{prop}

\begin{proof}The ``only if" part is proved in the previous proposition and we  need to show ``if" part.   
First  $d^*d\omega=\l \omega$ reads
\[
-*d*d\omega=\l \omega.
\]
We rewrite the above as,
\begin{equation}\label{E-s2}
d(*d\omega)=-\l\omega^2/2. 
\end{equation}
Note that $\p \omega=0$ means that $d\omega$ is the real part of a $(3, 0)$ form, denoted by $\Omega$. We write $Re(\Omega)=d\omega$. We have
\[
*\Omega=-\i \Omega,\; *\bar{\Omega}=\i \bar{\Omega}. 
\]
That is,
\[
*Re(\Omega)=Im(\Omega).
\]
By \eqref{E-s2}, we have
\[
d(Im(\Omega))=-\l\omega^2/2.
\]
Now suppose $\Omega=3\mu\tilde \Omega$ such that $\tilde \Omega$ has norm 1, for some nonzero constant $\mu$.  By by the equation we know that
\[
(d^*d\omega, \omega)=|d\omega|^2=\l |\omega|^2,
\]
hence $\l=12\mu^2$. Then we have
\[
d\omega=3\mu Re(\tilde \Omega), d(Im(\tilde \Omega))=-\frac{\l}{6\mu} \omega^2=-2\mu \omega^2. 
\]
This is the desired structure equation for a nearly Kahler structure in dimension 6. 
\end{proof}

We can see that $d^*d$-flow and $d^*d$-Ricci flow are both closely related to nearly K\"ahler structures and we believe it should give a plausible tool to find the existence of nearly K\"ahler structure.  We will discuss this aspect of the $d^*d$-flow and $d^*d$-Ricci flow in more details elsewhere.

\section{Smooth four manifolds}First we recall the definition of an \emph{irreducible} manifold.
\begin{defn}A smooth four-manifold $M$ is called irreducible if for every smooth decomposition of connected sums $M\approx M_1 \sharp M_2$, one summand $M_i$ must be homeomorphic to the standard four sphere $S^4$. 
\end{defn}

We consider simply-connected irreducible four manifold with an almost complex structure. When a manifold admits an almost complex structure  can be determined purely in terms of its cohomology. For a simply-connected four manifold, it admits an almost complex structure if and only if ${b_2}_{+}$ is odd. The known examples of simply-connected irreducible four manifolds all have almost complex structure. But in general this remains an open problem, \cite{Stern, GS}
\begin{q}\label{q-almost}Does every simply-connected irreducible four manifold with $b_2>0$ support an almost complex structure?
\end{q}

One of our main motivations is to understand the following,
 \begin{q}
 When a  simply-connected irreducible four manifold with an almost complex structure admits a symplectic structure? 
 \end{q}
 We believe this problem is one of the keys to reveal the mystery of geometry and topology of four manifolds. In this sense  the $d^*d$-flow  (and $d^*d$-Ricci flow) should have important applications. We run a $d^*d$-flow (or $d^*d$-Ricci flow) on such a manifold, and try to understand the geometry of topology of the underlying manifold by understanding the possible singularities and long time behavior of the flow.  If we are able to derive powerful analytic tools for these geometric flows, such as Hamilton-Perelman's theory for Ricci flow, there is a good chance to understand smooth four manifolds clearer. 

We can give some expectations, trying to to combine Seiberg-Witten invariants and possible exotic spheres together to govern the geometry of simply-connected four manifolds. Suppose two homeomorphic (simply connected irreducible four) manifolds $M, N$ have the same Seiberg-Witten invariants (whenever it is well defined), one might expect that either they are diffeomorphic to each other, or $M=N\sharp S$, for some exotic sphere $S$. We can be more specific for symplectic structures. 

\begin{q} \label{q-s1}Let $N$ be an irreducible almost complex simply-connected four manifold. Suppose $N$ is homeomorphic to a symplectic manifold $M$ and we assume further that the Seiberg-Witten invariant of $N$ is the same as $M$ (assume $b_+>1$), then either $N$ is diffeomorphic to $M$ or $N\approx M\sharp S$, where $S$ is a homotopy (exotic) four sphere. 
\end{q}

It is also interesting to try to relate the $d^*d$-flow (the singularities) with the surgeries extensively used in four manifolds to construct examples (see Conjectures in \cite{FS1} for example). 
Nevertheless,  we believe that the close study of $d^*d$-flow (as well as $d^*d$-Ricci flow) should help us to understand the smooth topology of four manifolds.
An ultimate understanding of this flow might eventually give a clearer understanding of smooth four manifolds. Indeed, S. T, Yau \cite{Yau} gave some suggestions to understand four manifolds:

``A true understanding of four manifolds probably should come from understanding the
question of existence of the integrable complex structures...A good conjectural statement need to be made on the topology of four
manifolds that may admit an integrable complex structure. Pseudo-holomorphic
curve and fibration by Riemann surfaces should provide important information.
Geometric flows may still be the major tool."

In the same vein we believe a good understanding of four manifolds could come from understanding the question of existence of symplectic structures among almost complex manifolds and $d^*d$-flow can be a very useful tool.

\end{document}